\documentclass{article}

\usepackage[ruled,longend,linesnumbered]{algorithm2e}
\SetAlFnt{\small}
\SetAlCapFnt{\small}
\SetAlCapNameFnt{\small}
\SetAlCapHSkip{0pt}
\IncMargin{-\parindent}

\usepackage{geometry}
\newgeometry{vmargin={24mm}, hmargin={33mm,53mm}}

\usepackage{graphicx}
\usepackage[caption = false]{subfig}

\usepackage{url}

\usepackage{lipsum}
\usepackage{amsfonts}
\usepackage{graphicx}
\usepackage{epstopdf}
\usepackage{algorithmic}
\usepackage{stmaryrd}
\usepackage{multirow}
\usepackage{amsmath}
\usepackage{amssymb}
\usepackage{mathtools}
\usepackage{pdfpages}
\usepackage{bbm}
\usepackage{diagbox}
\usepackage{amsbsy}
\usepackage{xcolor}
\usepackage{capt-of}
\usepackage{soul}
\usepackage{todonotes}
\usepackage{hyperref}
\usepackage{cleveref}
\usepackage{amsthm}

\definecolor{olive}{HTML}{bcbd22}
\definecolor{blue}{HTML}{1f77b4}
\definecolor{green}{HTML}{2ca02c}
\definecolor{red}{HTML}{d62728}
\definecolor{purple}{HTML}{9467bd}

\DeclareMathOperator*{\prox}{prox}
\DeclareMathOperator*{\argmin}{arg\,min}
\newcommand{\norm}[1]{\left\lVert#1\right\rVert}

\newcommand{\uv}{\mathbf{u}}
\newcommand{\x}{\mathbf{x}}
\newcommand{\dx}{\mathrm{d}}
\newcommand{\Vh}{\mathcal{V}_h}
\newcommand{\N}{\mathbb{N}}
\newcommand{\R}{\mathbb{R}}
\newcommand{\T}{T}
\newcommand{\M}{\mathbf{M}}
\newcommand{\D}{\mathbf{D}}
\newcommand{\A}{\mathbf{A}}
\newcommand{\K}{\mathbf{K}}
\newcommand{\p}{\mathbf{p}}
\newcommand{\Qh}{\mathcal{Q}_h}

\newcommand{\Qho}{\mathcal{Q}_h^1}
\newcommand{\Qht}{\mathcal{Q}_h^2}
\newcommand{\Qo}{\mathcal{Q}^1}
\newcommand{\Qt}{\mathcal{Q}^2}
\newcommand{\Fa}{F_h^{\alpha}}
\newcommand{\Fab}{\mathbf{F}_h^{\alpha}}

\newcommand{\Qcal}{\mathcal{Q}}
\newcommand{\Vcal}{\mathcal{V}}
\newcommand{\Xcal}{\mathcal{X}}
\newcommand{\Pf}{\mathcal{P}}
\newcommand{\TVSa}{\mathbf{TV\; S}^\alpha}
\newcommand{\TVSTa}{\mathbf{TV\; ST}^\alpha}
\newcommand{\TVSo}{\mathbf{TV\;S}^1}
\newcommand{\TVSTo}{\mathbf{TV\;ST}^1}
\newcommand{\TVSt}{\mathbf{TV\;S}^2}
\newcommand{\TVSTt}{\mathbf{TV\;ST}^2}

\newcommand{\db}{\mathrm{dB}}

\newcommand{\GT}{\mathbf{GT}}

\newcommand{\Tz}{\mathbf{T0}}
\newcommand{\To}{\mathbf{T1\;S}}
\newcommand{\TTo}{\mathbf{T1\;S+T}}
\newcommand{\TV}{\mathrm{TV}}
\newcommand{\tens}[1]{\otimes_{#1}}%

\newcommand{\Ft}{F_h^{2}}

\newcommand{\Fob}{\mathbf{F}_h^{1}}
\newcommand{\Fokb}{\mathbf{F}_{h,k}^{1}}
\newcommand{\Fodb}{\mathbf{F}_{h,d+1}^{1}}
\newcommand{\Ftb}{\mathbf{F}_h^{2}}

\usepackage[normalem]{ulem}
\newcommand{\soutr}{\bgroup\markoverwith{\textcolor{red}{\rule[0.5ex]{2pt}{1.2pt}}}\ULon}

\newtheorem{theorem}{Theorem}

\ifpdf
  \DeclareGraphicsExtensions{.eps,.pdf,.png,.jpg}
\else
  \DeclareGraphicsExtensions{.eps}
\fi

\title{Finite element-based space-time total variation-type regularization of the inverse problem in electrocardiographic imaging}
\author{Manuel Haas\footnotemark[2]
\and Thomas Grandits\footnotemark[3]
\and Thomas Pinetz\footnotemark[2]
\and Thomas Beiert\footnotemark[4]
\and Simone Pezzuto\footnotemark[5]
\and Alexander Effland\footnotemark[2]}

\begin{document}
	
\maketitle

\renewcommand{\thefootnote}{\fnsymbol{footnote}}

\footnotetext[1]{This work was supported by the Deutsche Forschungsgemeinschaft (DFG, German
Research Foundation), EXC2151-390873048 and EXC-2047/1-390685813. SP and TG are also supported by the SNSF project ``CardioTwin'' (no.~214817). SP acknowledges the support of the CSCS-Swiss National Supercomputing Centre project no.~s1074 and the PRIN-PNRR project no.~P2022N5ZNP. SP is member of INdAM-GNCS.}
\footnotetext[2]{Institute of Applied Mathematics, University of Bonn, Germany}
\footnotetext[3]{ Euler Institute, Universit\`{a} della svizzera italiana, Switzerland; Department of Mathematics and Scientific Computing, University of Graz, Austria}
\footnotetext[4]{ Heart Center Bonn, Department of Internal Medicine II, University Hospital Bonn, Germany}
\footnotetext[5]{ Department of Mathematics, University of Trento, Italy}

\renewcommand{\thefootnote}{\arabic{footnote}}

\begin{abstract}
Reconstructing cardiac electrical activity from body surface electric potential measurements results in the severely ill-posed inverse problem in electrocardiography.
Many different regularization approaches have been proposed to improve numerical results and provide unique results.
This work presents a novel approach for reconstructing the epicardial potential from body surface potential maps based on a space-time total variation-type regularization using finite elements, 
where a first-order primal-dual algorithm solves the underlying convex optimization problem.
In several numerical experiments, the superior performance of this method and the benefit of space-time regularization for the reconstruction of epicardial potential on two-dimensional torso data and a three-dimensional rabbit heart compared to state-of-the-art methods are demonstrated.
\end{abstract}

% REQUIRED
% \begin{keywords}
% cardiac electrophysiology; ECGI; inverse problem of electrocardiography; space-time total variation; space-time finite element
% \end{keywords}

% % REQUIRED
% \begin{MSCcodes}
% 65N21, 65N30, 65K10, 92C55, 49M29
% \end{MSCcodes}

\section{Introduction}
Deducing the electrical activity of the heart non-invasively has enormous potential to shorten clinical intervention times and improve patient outcomes. 
For this purpose, specialized hardware was developed to capture the electrical activity on the torso and reconstruct the epicardial potential~\cite{Cl18, Ra04} in the form of body potential surface maps (BSPM).
Reconstructing the epicardial potential from BSPMs by exploiting further information about the smoothness and regularity of the potential functions is widely known as the inverse problem in electrocardiographic imaging (ECGI).
The resulting epicardial potential enables clinicians to understand and identify the origins of arrhythmias.
Previous studies \cite{No22, Qu17} provide comprehensive overviews of inverse problems in cardiovascular modeling, including electrocardiographic imaging and blood flow modeling.
The mathematical formulation of the inverse problem in ECGI in terms of a linear forward operator is extensively studied in the literature (see e.g.~\cite{Fr14, Su06, Wa23}).
Computing the discretized forward operator numerically results in different choices like the boundary element method (BEM)~\cite{Ma10, ST08}, the method of fundamental solutions (MFS)~\cite{Wa06}, and the finite element methods (FEM)~\cite{Br08, Se05, WaKi11, Wa10}. 
In this work, we focus on the FEM, as it demonstrated improved performance for the solution of the inverse problem~\cite{Bo23}.

The potential-based ECGI problem results in an ill-posed Cauchy problem in the sense of Hadamard~\cite{Be07, Ha03, Ho09}.
In particular, the inverse solution is extremely sensitive to small perturbations of the measurements.
To overcome these issues, regularization of the epicardial potential is required to solve the inverse problem robustly.
Different approaches were proposed including the well-known Tikhonov regularization~\cite{Ti77} of different orders, model-based approaches~\cite{GeodesicBP2024,PezzutoBayes2022}, and machine learning-driven methods~\cite{Te22}.
Despite the fact that machine learning and deep learning approaches lead to promising results in simulated models, the training imposes challenges.
Due to the absence of available human data sets containing BSPMs, thoracic CT scans, and corresponding invasive measurements of epicardial potentials as ground truth reference, well-established methods like Tikhonov regularization are still among the most commonly applied types of regularization.
Tikhonov regularization imposes a penalty on the electrical activity of the heart or its derivatives with the quadratic $L^2$-norm causing considerable smoothing of the epicardial potential.
In some approaches (see e.g.~\cite{Ch03, Ka18}) different regularization techniques and hyperparameter choices for optimization are systematically compared.
Optimizing inverse problems with non-quadratic $L^1$ or $L^{2,1}$ regularization with $\norm{f}_{L^{2,1}(\Omega)}\coloneqq \int_{\Omega}\norm{f}_2\dx\x$ for $\Omega\subset\R^d$ open, $f\in (L^2(\Omega))^d$, and $d\in\mathbb{N}$ is widely known to successfully reconstruct solutions in imaging and computing less smoothed potential functions (see e.g.~\cite{Ba07, Ch10, Es04, Gh09, InvScar, Wa11}).

Further studies using temporal as well as spatial priors for spatiotemporal inverse problems, including Bayesian MAP-based regularization, show promising improvements (see e.g.~\cite{Um09, Gr03, Me04, On09}).

In this paper, we present a novel space-time total variation-type regularization in a finite element setting on unstructured grids that allows for sharp transitions of the electrical activity on the heart, which is crucial for accurately modeling rapid changes in cardiac electrical signals.
Unlike Tikhonov regularization, which primarily enforces smoothness and often struggles to represent sharp transitions effectively, total variation captures discontinuities and sharp gradients inherent in cardiac electrophysiology, making it a more suitable choice for this problem. 
The novelty of our approach lies in integrating space-time total variation regularization with finite element methods, a concept that has been scarcely explored in the literature. To identify the most suitable method, we introduce various finite element discretizations and different formulations of total variation for evaluation.
Here, the quality of the solution depends on two hyperparameters representing anisotropic smoothing in spatial and temporal directions. 
The proposed total variation-type regularizer is non-differentiable and therefore not amenable to classical derivative-based optimization algorithms. 
Classically, this non-differentiability can be overcome by a smooth approximation~\cite{Bo10}, which, however, neglects sharp interfaces promoted by total variation.
Thus, in this paper, we advocate a first-order primal-dual method~\cite{Cha11} adapted for finite elements, thereby exploiting the convexity of the energy function, resulting in fast and accurate reconstruction.
\Cref{fig:illustration} illustrates the basic concept of the inverse problem with total variation regularization.

For evaluation, we simulate the ground truth epicardial potential utilizing the pseudo bidomain equation~\cite{Bi11} on a simple two-dimensional torso model as well as on a three-dimensional rabbit heart~\cite{Mo22}.
Visual and quantitative comparisons are provided to demonstrate the superiority of our method over different state-of-the-art Tikhonov regularizations.
As the quality of the reconstruction is highly dependent on the number of measured potentials on the torso (see e.g.~\cite{Jo94}), we introduce a study of evaluation errors for an increasing number of torso electrodes.

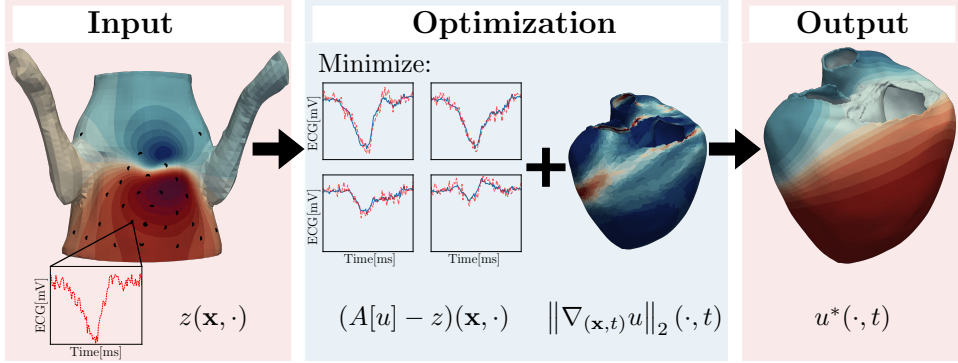
\begin{figure}[htp] \centering{
\begin{tikzpicture}
    \node[inner sep=0pt] (A) at (0,0){\includegraphics[clip, trim=0.7cm 0.3cm 1cm 0.1cm,width=1\textwidth]{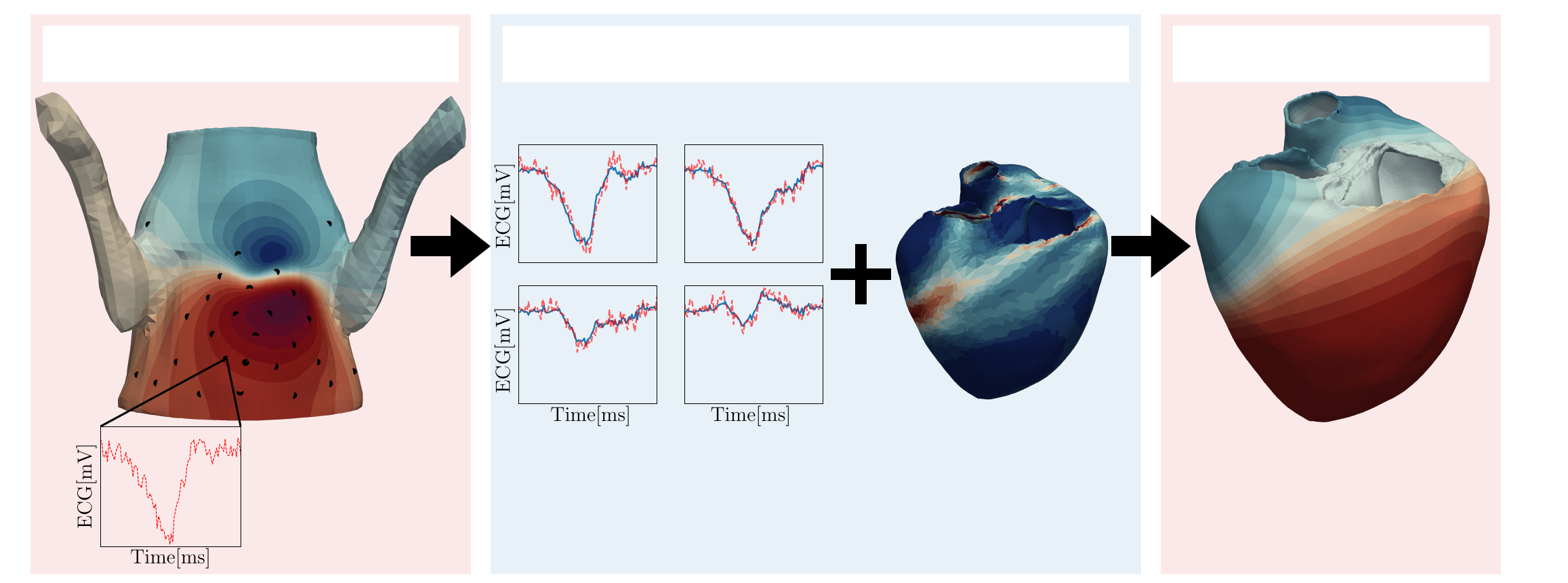}};
    \node[inner sep=0pt] at (-1.55,1.5) 
    {Minimize:};

    \node[inner sep=0pt] at (-4.8,2) 
    {{\large \textbf{Input}}};
    \node[inner sep=0pt] at (0.3,2) 
    {{\large \textbf{Optimization}}};
    \node[inner sep=0pt] at (4.8,2) 
    {{\large \textbf{Output}}};

    \node[] at (-3.7,-1.9) 
    {$z(\x,\cdot)$};
    \node[] at (-0.9,-1.9) 
    { $(A[u]-z)(\x,\cdot)$};
    \node[] at (1.9,-1.9) 
    { $\norm{\nabla_{(\x,t)}u}_{2}(\cdot,t)$};
    \node[] at (4.8,-1.9) 
    {$u^*(\cdot,t)$};

\end{tikzpicture}}
\caption{Basic illustration of total variation-based regularization for the inverse problem in ECGI on the rabbit database. Body surface measurements $z$ are compared to the optimization variable $u$ on the electrode domain $\Sigma$ and are minimized together with prior information in the form of the space-time gradient vector $\nabla_{(\x,t)}u$, resulting in the reconstruction $u^*$.}
\label{fig:illustration}
\end{figure}
Our work introduces the finite-dimensional inverse problem in ECGI in a continuous setting and applies finite element discretization to the non-smooth energy minimization problem. 
After reformulating to a saddle point problem, a primal-dual method is used to find an optimal solution.
The paper is structured as follows:
In \Cref{sec:space-timeECGI},  we present the forward and inverse problems in ECGI and prove the existence and uniqueness of a solution with the total variation-type regularizer.
Then, in \Cref{sec:discretization}, the discrete function spaces, the total variation operator, and the forward operator solving the forward problem are introduced along with their respective finite element discretization schemes.
In \Cref{sec:Primal-DualOptimization}, the primal-dual optimization algorithm tailored to the finite element setting is derived.
Finally, the numerical results are evaluated in \Cref{sec:numerical_results} for different noise variations in the two- and three-dimensional settings, where we additionally focus on the performance in an ablation study for an increasing number of thoracic potential measurements.

\section*{Notation}
For an open set $\mathcal{X}\in\R^N$ for $N\in\N$, we denote the Sobolev space by $H^s(\Omega)\coloneqq W^{s,2}(\Omega)$ with $s\in\R,s\geq 0$.
We define $\llbracket N\rrbracket\coloneqq\{1,\ldots, N\}$ and $\llbracket N\rrbracket_0\coloneqq\{0,\ldots, N\}$ for any $N\in\N$.
Real-valued vectors of dimension $N\cdot M$ are defined by bold lowercase letters $\mathbf{x}\in\R^{NM}$ and real-valued matrices by bold capital letters $\mathbf{X}\in\R^{N M\times N M}$ for $N,M\in\N$.
The maximum eigenvalue of a matrix $\mathbf{X}$ is denoted as $\lambda_{\mathrm{max}}(\mathbf{X})$.
Furthermore, $\Xcal^h$ refers to a tesselation for any set $\Xcal\subset \R^N$.
The continuous identity function is defined as $I_h$ and the $N$-dimensional identity matrix as $\mathbf{I}_N$.
We denote by $\mathrm{diag}(a_i)_{i=1}^N$ the diagonal matrix in $\R^{N\times N}$ with diagonal elements $a_1,\ldots,a_N\in\R$ and $0$ as off-diagonal entries.
The tensor product is defined as $\otimes$; the Kronecker product of two matrices is denoted by $\tens{K}$.
We denote the Dirac delta function of a set $\Xcal$ by $\delta_{\Xcal}$.

\section{Forward and Inverse Problem in Electrocardiographic Imaging}\label{sec:space-timeECGI}
In the following section, we introduce the forward problem in electrocardiographic imaging and deduce the inverse problem.
Furthermore, we propose a novel space-time total variation-type regularizer and prove the existence and uniqueness of weak solutions.

\subsection{Forward Problem}
\label{subsec:forward_problem}
The well-posed forward problem in ECGI aims to compute the electrical activity on the torso domain from the given epicardial potential.
Let us denote the \emph{torso domain} by $\Omega\subset\R^d$ for $d\in\{2,3\}$ with \emph{body surface} $\Gamma = \partial \Omega$ and associated \emph{unit outer normal vector} $\mathbf{n}$, and the \emph{heart} by $\Omega_H\subset\Omega$ with \emph{epicardium} $\Gamma_H = \partial \Omega_H$.
Furthermore, we denote by $\Omega_0 = \Omega\setminus\bar{\Omega}_H$ the torso domain excluding the cardiac region.
The boundary $\partial \Omega_0 = \Gamma\cup\Gamma_H$ is assumed to be Lipschitz and bounded, and $\Gamma \cap \Gamma_H = \emptyset$.
In our work, we exclusively assume static domains that are invariant over time.
We also denote by $\sigma\in L^\infty(\Omega, \R^+)$ the \emph{torso's conductivity tensor (bulk conductivity)}, where we assume the space-dependent matrix $\sigma$ to be symmetric satisfying the ellipticity condition~$\zeta^{-1}|\mathbf{y}|^2\leq \sigma(\x)\mathbf{y}\cdot \mathbf{y}\leq \zeta|\mathbf{y}|^2$ for some $\zeta > 0$ and all $\mathbf{y} \in \R^d$.
Then, the \emph{volumetric torso potential} $v:\Omega_0\times\T\to\R$ in some finite time interval $\T=(0,\bar{t})$ for $\bar{t}>0$ with given \emph{epicardial potential} $u:\Gamma_H\times\T\to\R$ solves the following (forward) problem 
\begin{equation}
\label{eq:E_forward}
\begin{cases}
\begin{aligned}
-\operatorname{div}\left(\sigma(\x)\nabla_{\x}v(\x,t)\right) &= 0,  &&\left(\mathbf{x},t\right)\in \Omega_0\times\T, \\
\sigma(\x)\nabla_{\x}v(\x,t)\cdot\mathbf{n} &= 0,  &&\left(\mathbf{x},t\right)\in \Gamma\times\T,  \\
v(\x,t) &= u(\x,t),  &&\left(\mathbf{x},t\right)\in \Gamma_H\times\T.
\end{aligned}
\end{cases}
\end{equation}
The subsequent theorem ensures the existence, uniqueness, and regularity of solutions of this elliptic PDE system.
\begin{theorem}
There exists a unique solution $v\in H^1(\T,H^1(\Omega_0))$ of \eqref{eq:E_forward} for $u\in H^1(\T,H^{1/2}(\Gamma_H))$ in a weak sense.
\end{theorem}
\begin{proof}
The proof relies on standard arguments from elliptic PDE theory, regularity theory, and the trace theorem; we omit further details and refer the reader to~\cite{Ev10, Ne12}.
\end{proof}
We remark that, under suitable regularity assumptions on the temporal part of the function,
the existence and uniqueness of a solution to the space-time problem is ensured by pointwise evaluation in time for the boundary function $u$.
Let $v_u$ be the solution associated with $u$.
Then, we can define the continuous \emph{forward operator}
\[
A:H^1(\T, H^{1/2}(\Gamma_H))\to H^1(\T, H^{1/2}(\Gamma)), \quad u\mapsto v_u|_\Gamma, 
\]
which ``transfers'' the epicardial potential to the torso.
Note that the inverse of the forward operator is, in general, unbounded.

\subsection{Inverse Problem}
In this subsection, we introduce the severely ill-posed inverse problem of reconstructing the epicardial potential using measurements $z$ on the body surface.
In detail, we assume the \emph{body surface potential} $z\in L^2(\Sigma\times\T)$ at the electrodes represented by $\Sigma \subset \Gamma$ for $N_\Sigma\in\N$ electrodes to be known, and no further information about the potential on the heart is available.
Henceforth, we assume that the electrodes do not collapse to a single point, reflected by a positive $(d-1)$-dimensional surface measure of the electrodes.
The setting is illustrated in \Cref{fig:2D_model}.
\begin{figure}[htp] \centering{
\includegraphics[clip, trim=0cm 1.6cm 0cm 0.5cm,width=0.5\textwidth]{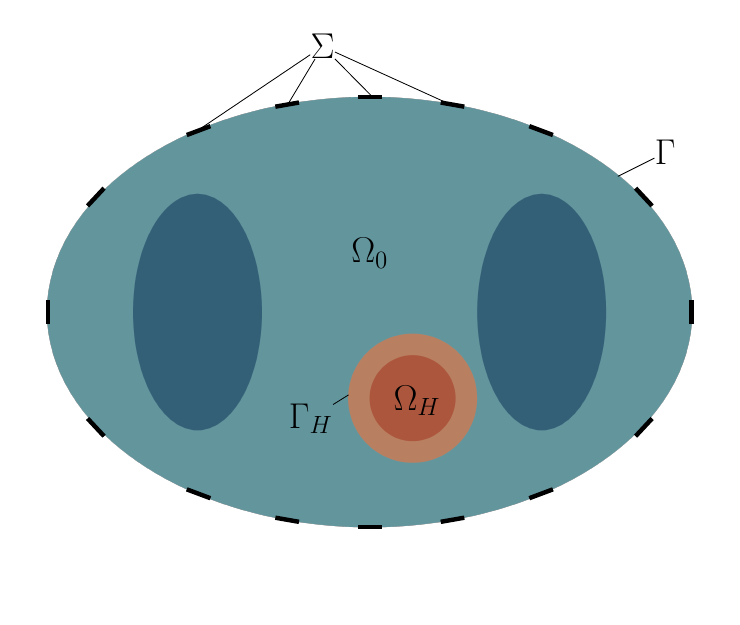}}
\caption{Two-dimensional torso-heart model with heart domain $\Omega_H$, torso domain $\Omega_0$ including lungs, epicardium $\Gamma_H$, torso boundary $\Gamma$ and $16$ body surface electrodes $\Sigma$.}
\label{fig:2D_model}
\end{figure}
For the inverse problem, we assume Dirichlet boundary conditions on the torso domain instead of Dirichlet boundary conditions on the heart as in \eqref{eq:E_forward}.
Consequently, we aim to minimize the cost function
\begin{equation}\label{eq:G_cont}
\argmin_{u\in H^{1}\left(\Gamma_H\times\T\right)} 
\left\{
G(u)\coloneqq\frac{1}{2 |\Sigma | N_\Sigma}\int_{\Sigma\times\T}\left(A[u](\x,t)-z(\x, t)\right)^2 \dx (\x,t)
\right\},
\end{equation}
where $|\Sigma|$ is the constant area of the electrodes $\Sigma$.
Driven by the ill-posedness, we follow the common paradigm in inverse problem theory by adding a regularization term.
Various methods for finding the best solution to this problem have been proposed such as zero- or first-order Tikhonov regularization methods promoting sparsity or, more recently, machine learning-based techniques.
It is well-known that $ L^2$-norm-based regularizers exhibit smoothing effects while preventing sharp transitions of solutions.
In a seminal paper, Rudin, Osher, and Fatemi~\cite{Ru92} advocated the total variation regularizer for applications in imaging to overcome the aforementioned issues.
Promising studies about total variation-type regularization (in space only) for the inverse problem in ECGI have already been conducted in a BEM setting (see e.g.~\cite{Gh09, Sh10, Ji14}).
Here, the normal derivative of the epicardial potential is penalized in the $L^1$-norm.
To the best of our knowledge, we are the first to introduce a finite-element total variation-type approach in both space and time for the inverse problem in ECGI.
We here advocate the spatiotemporal total variation-type regularizer with $L^1$- or $L^{2,1}$-norm given by
\begin{equation}\label{eq:F_cont}
    F^{\alpha}(u) = \int_{\Gamma_H\times\T}\TV_\epsilon^\alpha\left(Ku(\x,t)\right)\dx (\x,t) 
\end{equation}
$\alpha\in\{1,2\}$ for each of the norms respectively, and the weighted total variation operator $K \coloneqq \Lambda\nabla_{(\x,t)}$. 
The classical total variation is extended by an anisotropy term represented by a matrix $\Lambda = \mathrm{diag}(\lambda_{\gamma},\ldots,\lambda_{\gamma},\lambda_t)\in\R^{(d+1)\times (d+1)}_+$, where $\lambda_{\gamma}$ and $\lambda_t$ represent the spatial and temporal penalization parameters, respectively.
We consider the subsequent half-quadratic penalization functions for $\x\in\R^{d+1}$ and fixed $\epsilon>0$
\begin{align*}
\TV_\epsilon^1(\x)&=\sum_{i=1}^{d+1}\widetilde{\TV}_\epsilon^1(x_i), \quad \text{with} \quad
\widetilde{\TV}_\epsilon^1(x_i) = \begin{cases}
|x_i|, &|x_i|\leq\frac{1}{\epsilon},\\
\epsilon x_i^2,  &\text{else},
\end{cases}
\quad \text{and} \\
\TV_\epsilon^2(\x)&=
\begin{cases}
\norm{\x}_2, &\norm{\x}_2\leq\frac{1}{\epsilon},\\
\epsilon\norm{\x}_2^2,  &\text{else}.
\end{cases}
\end{align*}
By considering $\TV_\epsilon^\alpha$ we can analyze the function in classical Sobolev spaces rather than in the space of functions of bounded variation, which facilitates the analytical treatment of the problem.
Then, we can define the \emph{regularized energy function} for $u\in H^{1}(\Gamma_H\times\T)\subset H^1(\T,H^{1/2}(\Gamma_H))$ as $J^{\alpha}(u)=F^{\alpha}(u)+G(u)$.
\begin{theorem} 
$J^{\alpha}+T$ for a Tikhonov regularizer $T(u)=\frac{\eta}{2}\int_{\Gamma_H\times\T}u(x,t)^2\dx (\x,t)$ with $\eta>0$ admits a unique minimizer in $H^{1}(\Gamma_H\times\T)$ for $\alpha\in\{1,2\}$.
\end{theorem}
\begin{proof}
We employ the direct method in the calculus of variations. 
Note that the additional Tikhonov regularizer~$T$ is solely required to ensure coerciveness and dropped in all further considerations.
Let $\{u_n\}_{n\in\N}\subset H^{1}(\Gamma_H\times\T)$ be a minimizing sequence, i.e.
$J^{\alpha}(u_1)\geq J^{\alpha}(u_n)\to\inf_{u\in H^1(\Gamma_H\times\T)} J^{\alpha}(u)$.
For every thoracic potential $z \in L^2(\Gamma\times\mathbb T)$, the associated sequence $J^{\alpha}(u_n)$ is constrained from below by $0$ and from above due to the boundedness of the forward operator $A$.
By the definition of $\TV_\epsilon^\alpha$ we directly obtain that there exists $C>0$ such that
$J^{\alpha}(u_1)\geq J^{\alpha}(u_n)\geq C \norm{u_n}^2_{H^1(\Gamma_H\times\T)}$
for all $n\in\mathbb{N}$.
Since $H^1(\Gamma_H\times\T)$ is Hilbert, there exists a subsequence converging weakly to $u^\ast\in H^1(\Gamma_H\times\mathbb \T)$.
Due to the linearity of $A$ and the convexity of $J$ we can deduce that $J^{\alpha}$ is sequentially lower semi-continuous, which implies $J^{\alpha}(u^\ast) \leq \inf_{u\in H^1(\Gamma_H\times\T)} J^{\alpha}(u)$. 
Since $A$ is injective, $G$, and consequently $J^\alpha$, is strictly convex. 
Hence, the minimizer is unique.
\end{proof}

\section{Discretization}\label{sec:discretization}
In this section, we propose a spatiotemporal discretization scheme for the inverse problem in ECGI introduced in \Cref{sec:space-timeECGI}.

\subsection{Discrete Function Spaces}
In what follows, we briefly present the discrete function spaces used throughout this manuscript using the finite element method  (FEM). 
The discrete space-time function space $\Vh$ approximates $H^1\left(\Gamma_H\times\T\right)$ with $\Pf_1$ finite elements in space on an adaptive surface grid $\mathcal{T}_h$ and in time on a grid $\mathcal{S}_h$ with grid size function $h$.
In detail, we assume that measurements are available at $S+1$ time points $0=t_0<\ldots<t_S=\bar{t}$ admitting $S$ time intervals.
For each time point, the spatial surface grid is identical and comprises affine ($d=2$)/triangular ($d=3$) elements.
We denote by $\pmb{\Xi}\coloneqq\left\{\pmb{\xi}_i\right\}_{i=1}^{N_{\Vcal}}$ the nodes of $\mathcal{T}_h$ that are contained in $\Gamma_H$ and by $\Gamma^h_H$ the set of all elements for which the corresponding nodes are all contained in $\pmb{\Xi}$.
Likewise, $\T^h$, $\Omega_0^h$, and $ \Gamma^h$ describe the discretizations of $\T$, $\Omega_0$ and $\Gamma$, where
we tacitly assume that each connected component of $\Sigma^h$ is comprised of at least one node of $\Gamma^h$.
Combining the spatial elements $L\in\mathcal{T}_h$ and temporal elements $J\in\mathcal{S}_h$ with the tensor product results in rectangular ($d=2$)/prismatic ($d=3$) elements $L\otimes J\in\mathcal{T}_h\otimes \mathcal{S}_h$.
Consequently, the discrete function space $\Vh$ is defined by
\begin{equation*}
\Vh \coloneqq \{v\in H^1\left(\Gamma_H\times\T\right)\colon
v|_{J\otimes K}\in \Pf_1(L)\otimes \Pf_1(J)\quad
\forall L \in\mathcal{T}_h, J\in \mathcal{S}_h \}
\end{equation*}
with piecewise affine spatial and temporal basis functions $\left\{\varphi_i\right\}_{i=1}^{N_{\Vcal}}$ and $\left\{\rho_s\right\}_{s=0}^{S}$, respectively.
The discrete function $u_h\in\Vh$
\[
u_h(\x,t)\coloneqq\sum_{s=0}^{S}\sum_{i=1}^{N_{\Vcal}} u_{i,s}\varphi_i(\x)\rho_s(t)
\]
with $u_{i,s}$ being the evaluation of $u_h$ at the node $(\pmb{\xi}_i,t_{s})$.
The associated discrete gradient space is induced by the integration rules of the discrete gradient applied in the optimization algorithm.
Since differentiating $\Pf_1$ functions results in $\Pf_0$ functions, applying a spatial and temporal gradient induces a gradient space with mixed $\Pf_1$ and $\Pf_0$ basis functions.
The gradient space $\Qho$ is composed of all vector-valued functions $p_h=(p_{h_1},\ldots,p_{h_{d+1}})^\top$ with $p_{h_k}\in L^2(\Gamma_H\times\T)$ for $k\in\llbracket d+1\rrbracket$, which are piecewise constant in space and affine in time in the first $d$~components representing the spatial gradients and piecewise constant in time and affine in space in the last component representing the temporal gradient, i.e.
\begin{equation*}
\Qho \coloneqq \left\{ p\in (L^2\left(\Gamma_H\times\T\right))^{d+1} \colon\begin{array}{l}
    p_k|_{L\otimes J}\in \Pf_0(L)\otimes \Pf_1(J),k\in\llbracket d\rrbracket\\
    p_{d+1}|_{L\otimes J}\in \Pf_1(L)\otimes \Pf_0(J),
\end{array} L\in\mathcal{T}_h, J\in \mathcal{S}_h\right\}
\end{equation*}
with piecewise constant polynomials $\{\vartheta_l\}_{l=1}^{N_{\Qcal}}$ for $N_{\Qcal}\in\N$ and $\{\varrho_j\}_{j=1}^{S}$ representing a $\Pf_0$ finite element basis of $\Gamma_H^h$ and $\T^h$, respectively.
First-order integration rules are applied to compute the regularizer function $F^\alpha$.
We choose a first-order quadrature rule with evaluation at the nodes of each spatial $\int_L f(\x)\dx\x\approx |L|/d \sum_{i=1}^d f(\pmb{\xi}_i)$ for $d\in\{2,3\}, L\in\mathcal{T}_h$ and temporal element $\int_J f(t)\dx t\approx |J|/2 (f(t_{s}) + f(t_{s-1}))$ for $J\in\mathcal{S}_h$.

The optimization involving space and time gradients in the $L^{2,1}$-norm poses challenges due to the continuity of the gradients in the complementary dimension, which results in interdependencies.
Since conformal computation of the gradients is required, optimizing discretized functions on space-time elements $L\otimes J$ element-wise, independent of neighboring element,s can be beneficial for the evaluation of the gradients.
However, element-wise independent optimization does not induce a finite element space due to discontinuities between adjacent elements.
Such a space $\Qht$ consists of elementwise continuous functions in space and time
\begin{equation*}
\Qht \coloneqq \{p\in (L^2\left(\Gamma_H\times\T\right))^{d+1}\colon
p_k|_{L\otimes J}\in \Pf_1(L)\otimes \Pf_1(J), k\in \llbracket d+1 \rrbracket, L \in\mathcal{T}_h, J\in \mathcal{S}_h\}.
\end{equation*}
An illustration of the spatial and temporal gradients on a prismatic finite element for both gradient spaces can be seen in \Cref{fig:quadratures}.
In this paper, we consider the $\TV_\epsilon^\alpha$ method for $\alpha=1$ in the gradient space $\Qho$ and for $\alpha=2$ in $\Qht$, even though  $\alpha=1$ would also be possible for the gradient space $\Qht$.
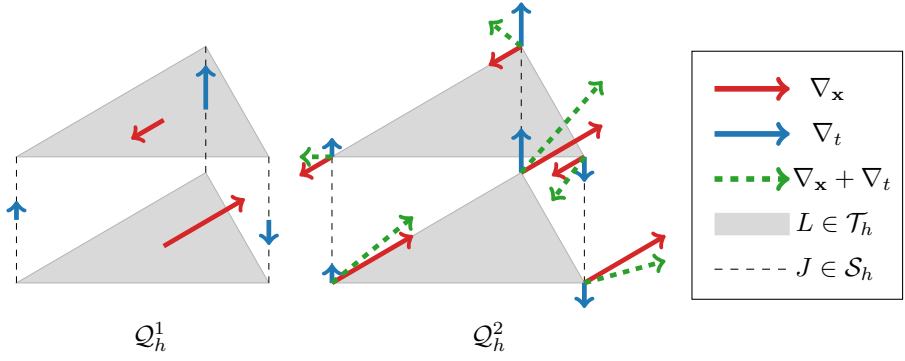
\begin{figure}[htp]
\label{fig:quadratures}\centering{
\begin{tikzpicture}
    \node[inner sep=0pt] (A) at (0,0){\includegraphics[clip, trim=0cm 0.2cm 0cm 0cm,width=0.7\textwidth]{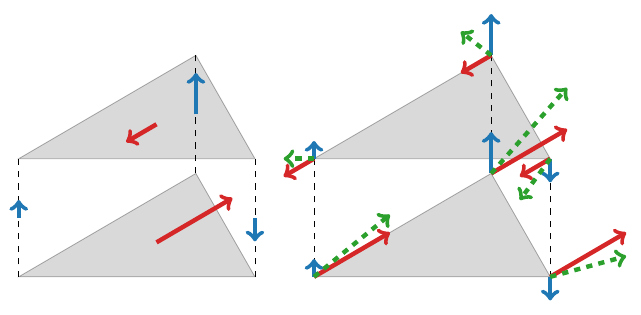}};
\node[inner sep=0pt] at (-2.5,-2.5) {$\Qho$};
\node[inner sep=0pt] at (2,-2.5) {$\Qht$};

\draw[red,->, line width=0.75mm] (5,0.8) -- (6,0.8);
\draw[blue,->, line width=0.75mm] (5,0.2) -- (6,0.2);
\draw[green,->, dashed, line width=0.75mm] (5,-0.4) -- (6,-0.4);

\node[inner sep=0pt] at (6.5,0.8) {$\nabla_\x$};
\node[inner sep=0pt] at (6.5,0.2) {$\nabla_t$};
\node[inner sep=0pt] at (6.7,-0.4) {$\nabla_\x + \nabla_t$};

\draw[gray, opacity = 0.3, line width = 0.3 cm] (5,-1) -- (6,-1);
\draw [dashed] (5,-1.6) -- (6,-1.6);

\node[inner sep=0pt] at (6.6,-1) {$L\in\mathcal{T}_h$};
\node[inner sep=0pt] at (6.6,-1.6) {$J\in\mathcal{S}_h$};

\draw  (4.7,-2) rectangle (7.5,1.3);

\end{tikzpicture}}
\caption{Basic illustration of the gradient computation on a prismatic finite element $L\otimes J\in\mathcal{T}_h\otimes \mathcal{S}_h$ in spatial dimension $d=3$ for different gradient spaces $\Qho$ and $\Qht$.}
\end{figure}

The discretized functions are defined by the basis functions of the finite element spaces and a \emph{vector of nodal values}.
To simplify notation, we iterate over the indices of vectors for each of their dimension separately and define a vector of nodal values for $u_h\in\Vh$ as $\mathbf{u} = (u_{1,0},\ldots,u_{N_{\Vcal},0},u_{1,1},\ldots,u_{N_{\Vcal},S})^\top\in\R^{N_{\Vcal} (S+1)}$.
Following the definition of $p_h\in\Qho$ as a vector of $d+1$ functions, the \emph{vector of element values} is defined as a concatenation of vectors $\p = \left(\p_{1},\ldots,\p_{d+1}\right)^\top$ for all spatial derivatives with $\p_{k}\in\R^{N_{\Qcal}(S+1)}$  for $k\in\llbracket d\rrbracket$ and $\p_{d+1}\in\R^{N_{\Vcal}S}$.
Note that the vector of element values for $p_h\in\Qht$ is an element of the space $\R^{(d+1) N_{\Qcal^2}}$ with $N_{\Qcal^2}\coloneqq 2\cdot d \cdot N_\Qcal\cdot S$ representing $d+1$ gradients on $2\cdot d$ evaluation nodes of the $N_\Qcal\cdot S$ rectangular/prismatic elements of $\mathcal{T}_h\otimes \mathcal{S}_h$ as illustrated in \Cref{fig:quadratures}.
We compute scalar products in matrix-vector notation by defining the 
spatial and temporal mass matrices $\M^{\Pf_1}$ and $\D^{\Pf_1}$ for $\Vh$ as follows:
\begin{equation*}
\M^{\Pf_1} = \left(\int_{\Gamma_{H}^h}\varphi_i(\x)\varphi_j(\x)\dx \x\right)_{i,j=1}^{N_{\Vcal}}, \quad \D^{\Pf_1} = \left(\int_{\T^h}\rho_i(\x)\rho_j(\x)\dx \x\right)_{i,j=0}^{S}.
\end{equation*} 
The respective spatial and temporal mass matrices $\M^{\Pf_0}$ and $\D^{\Pf_0}$ for the weights of the gradient spaces read as
\begin{equation*}
\M^{\Pf_0} = \mathrm{diag}\left(\int_{\Gamma_{H}^h}{\vartheta}_{i}(\x)\dx\x\right)_{i=1}^{N_{\Qcal}}, \quad \D^{\Pf_0} = \mathrm{diag}\left(\int_{\T^h}{\varrho}_{i}(\x) \dx\x\right)_{i=1}^{S}.
\end{equation*}
The corresponding block diagonal matrices for the number of timesteps and different quadrature rules are defined with the help of the Kronecker product $\tens{K}$ for $\Vh$ and the gradient spaces $\Qho$ and $\Qht$ as
\begin{align*}
   &\M^{\Vcal} = \mathbf{I}_{S+1}\tens{K}\M^{\Pf_1},
    \quad \M^{\Qo,\Pf_1}= \mathbf{I}_{S}\tens{K}\M^{\Pf_1},\\
    &\M^{\Qo,\Pf_0}= \mathbf{I}_{S+1}\tens{K}\M^{\Pf_0},\quad \M^{\Qt}= \mathbf{I}_{(d+1)\frac{N_{\Qcal^2}}{N_{\Qcal}}}\tens{K}\frac{1}{d}\M^{\Pf_0}.
\end{align*}
Since $\D^{\Pf_1}$ is a tridiagonal matrix, the temporal block mass matrices induce weights in block tridiagonal matrices for $\Pf_1$ finite elements and block diagonal matrices for $\Pf_0$ finite elements, i.e.
\begin{align*}
&\D^{\Vcal}  = \D^{\Pf_1}\tens{K} \mathbf{I}_{N_\Vcal}, \quad \D^{\Qo,\Pf_1}  = \D^{\Pf_1}\tens{K} \mathbf{I}_{N_\Qcal},\quad \D^{\Sigma}  = \D^{\Pf_1}\tens{K} \mathbf{I}_{N_\Sigma},\\
&\D^{\Qo,\Pf_0}  = \D^{\Pf_0}\tens{K} \mathbf{I}_{N_\Vcal},\quad \D^{\Qt} = \frac{1}{2}\D^{\Pf_0} \tens{K} \mathbf{I}_{(d+1)\frac{N_{\Qcal^2}}{S}}.
\end{align*}
We want to point out that all mass matrices are symmetric positive definite by construction since the Kronecker product of two symmetric positive definite matrices is symmetric and positive definite.
Furthermore, the spatial and temporal mass matrices in the same spaces commute by the \emph{mixed-product property} since 
\begin{equation*}
(\mathbf{X}\tens{K} \mathbf{I}_N)(\mathbf{I}_M\tens{K} \mathbf{Y})=(\mathbf{X}\mathbf{I}_M)\tens{K}(\mathbf{I}_N\mathbf{Y})=(\mathbf{Y}\tens{K}\mathbf{I}_M)(\mathbf{I}_N\tens{K} \mathbf{X})
\end{equation*}
for $(\mathbf{X},\mathbf{Y})\in\R^{M\times M}\times\R^{N\times N}$ and $N, M\in\N$.
The corresponding lumped mass matrices of $\M^\Xcal$ and $\D^\Xcal$ for a set $\Xcal$ are computed by
\begin{equation*}
\widetilde{\M}^\Xcal \coloneqq \mathrm{diag}\left(\sum_j \M^{\Xcal}_{i,j}\right)_i, \quad \text{and} \quad \widetilde{\D}^\Xcal \coloneqq \mathrm{diag}\left(\sum_j \D^{\Xcal}_{i,j}\right)_i.
\end{equation*}
For the quadrature rules in $\Qho$ in the regularizer function $F^1$, we define the vectors of nodal values in time $\mathbf{d}^{\Pf_1}\in\R^{S}$ and space $\mathbf{m}^{\Pf_1}\in\R^{N_{\Vcal}}$ for the first-order quadrature rule as follows: 
\begin{align*}
d^{\Pf_1}_s &= \frac{1}{2}\sum_{j=1}^{2} |J_j^s|, \quad s\in\llbracket S\rrbracket_0\quad \text{with} \quad t_s\in J_j^s\quad \text{for} \quad J_j^s\in\mathcal{S}_h,\\
m^{\Pf_1}_i &= \frac{1}{d}\sum_{l=1}^{N_{\pmb{\xi}_i}}\vert L_l^i\vert, \quad i\in\llbracket N_{\Vcal}\rrbracket\quad \text{with} \quad \pmb{\xi}_i\in L_l^i\quad \text{for} \quad \pmb{\xi}_i\in \pmb{\Xi}, L_l^i\in\mathcal{T}_h,
\end{align*}
with $N_{\pmb{\xi}_i}$ representing the number of elements for which $\pmb{\xi}_i$ is part of.
The induced $L^2$-inner product of $\Vh$ results in
\begin{align*}
    (u_h,v_h)_{\Vh} &=\int_{\Gamma_H^h \times \T^h} u_h(\x,t) v_h(\x,t)\dx(\x,t)    = \uv^\top  \D^{\Vcal} \M^{\Vcal}\mathbf{v}.
\end{align*}
The standard $L^2$ scalar product for the space $\Qho$ is split into the spatial and temporal gradient parts with different weights
\begin{align*}
    \left(p_h, q_h\right)_{\Qho} &=\int_{\Gamma_H^h \times \T^h} p_h(\x,t) \cdot q_h(\x,t)  \mathrm{d}(\x,t) \\
&=\sum_{k=1}^d\p_{k}^\top  \D^{\Qo,\Pf_1} \M^{\Qo,\Pf_0}\mathbf{q}_{k} + \p_{d+1}^\top  \D^{\Qo,\Pf_0} \M^{\Qo,\Pf_1}\mathbf{q}_{d+1}
\end{align*}
and for $\Qht$ as
\begin{align*}
    \left(p_h, q_h\right)_{\Qht} &=\int_{\Gamma_H^h \times \T^h} p_h(\x,t) \cdot q_h(\x,t)  \mathrm{d}(\x,t) =\p^\top  \D^{\Qt} \M^{\Qt}\mathbf{q}.
\end{align*}

\subsection{Total Variation Operator}
We introduce the discretized weighted total variation operator $K_h$ as a splitting of the gradient $\nabla_{(\x,t)}$ in the spatial and the temporal part weighted by the matrix $\Lambda$.
The spatial part computes $d$-dimensional directional derivatives $\nabla_\x$ on the heart surface incorporating a $\Pf_1$ element basis (see \cite{Ch10}).
For each time interval, a temporal derivative is computed in a $\Pf_1$ finite element basis.
The weighted total variation operator $K_h$ is defined as $K_h u_h\coloneqq(\lambda_{\gamma}\nabla_\x u_h,\lambda_t\nabla_t  u_h)^\top$.
Since the finite element spaces are fixed for each weighted total variation operator, the matrix representations $\K^\alpha$ of $K_h$ for $\alpha\in\{1,2\}$ applied on the vector of nodal values $\mathbf{u}\in \R^{N_{\Vcal}(S+1)}$ of $u_h\in \Vh$ is related to the corresponding gradient space $\Qh^\alpha$.
The matrix representation $\K^1 =(
\hat{\K}^1,\ldots, \hat{\K}^{d+1})^\top$
induced by the gradient space $\Qho$ is divided into space gradient matrices $\hat{\K}^k: \R^{N_{\Vcal}(S+1)}\to \R^{N_{\Qcal}(S+1)}$ with $k\in\llbracket d \rrbracket$ and a time gradient matrix $\hat{\K}^{d+1}: \R^{N_{\Vcal} (S+1)}\to \R^{N_{\Vcal} S}$ via
\begin{equation*}
(\hat{\K}^k \mathbf{u})_{l,j}  = \lambda_{\gamma}\sum_{s=0}^{S}\sum_{i=1}^{N_{\Vcal}}K^k_{l,j,i,s}u_{i,s} = \lambda_{\gamma}\sum_{i=1}^{N_{\Vcal}}\nabla_{x_k}\varphi_i(\x)|_{L_l}u_{i,j}
\end{equation*}
for $L_l\in\mathcal{T}_h$ and $(l,j)\in\llbracket N_{\Qcal}\rrbracket\times\llbracket S\rrbracket_0$.
The time gradient matrix is defined as 
\begin{equation*}
(\hat{\K}^{d+1} \mathbf{u})_{l,j} = \lambda_t\sum_{s=0}^{S}\sum_{i=1}^{N_{\Vcal}}K^{d+1}_{l,j,i,s}u_{i,s}= \lambda_t\sum_{s=0}^S \nabla_t \rho_{s}(t)|_{J_j} u_{l,j}
\end{equation*}
for $J_j\in \mathcal{S}_{h}$ and $(l,j)\in\llbracket N_{\Vcal}\rrbracket\times\llbracket S\rrbracket$.

By definition of the gradient space $\Qht$, the weighted total variation  operator $\K^2$ computes the gradients and interpolates them elementwise with a matrix $\mathbf{P}$ to the nodes of each element such that $\K^2\coloneqq \mathbf{P}\cdot \K^1:\R^{N_{\Vcal}(S+1)}\to \R^{(d+1)N_{\Qcal^2}}$.
The adjoint~$K_{h}^*$ of the weighted total variation operator $K_{h}$ is implicitly defined by  
\begin{equation*}
    \left(K_{h} u_h,  p_h\right)_{\Qh^\alpha} =\left(u_h,  K_{h}^* p_h\right)_{\Vh}\quad \text{with} \quad \left(u_h,  K_{h}^* p_h\right)_{\Vh} = \mathbf{u}^\top\D^{\Vcal} \M^{\Vcal}(\K^\alpha)^* \p
\end{equation*}
for $\alpha\in\{1,2\}$.
By definition of the scalar product in $\Qho$, we can compute
\begin{equation*}
    \left(K_{h} u_h,  p_h\right)_{\Qho} = \sum_{k=1}^d(\hat{\K}^{k}\mathbf{u})^\top  \D^{\Qo,\Pf_1} \M^{\Qo,\Pf_0}\p_{k}+ (\hat{\K}^{d+1}\mathbf{u})^\top  \D^{\Qo,\Pf_0} \M^{\Qo,\Qo,\Pf_1}\p_{d+1}.
\end{equation*}
Thus, the adjoint weighted total variation operator results in
\begin{equation*}
    (\hat{\K}^k)^* = (\M^{\Vcal})^{-1} (\hat{\K}^k)^\top \M^{\Qo,\Pf_0} \quad \text{and}\quad (\hat{\K}^{d+1})^* = (\D^{\Vcal})^{-1} (\hat{\K}^{d+1})^\top \D^{\Qo,\Pf_0}.
\end{equation*}
Likewise, the adjoint of $\K^2$ in the gradient space $\Qht$ satisfies
\begin{equation*}
    \left(K_{h} u_h,  p_h\right)_{\Qht} = (\K^2 \mathbf{u})^\top  \D^{\Qt} \M^{\Qt}\mathbf{q},
\end{equation*}
resulting in
\begin{equation*}
    (\K^2)^* = (\D^{\Vcal})^{-1} (\M^{\Vcal})^{-1} (\K^2)^\top \D^{\Qt} \M^{\Qt}.
\end{equation*}

\subsection{Forward Operator}
Next, we discretize the forward operator $A$ in the finite element setting following~\cite{Wa10}.
Since the epicardium is assumed to be rigid over time, we can compute the forward operator for fixed time $\tilde{t}\in\T$, i.e. $u_h(\x,\tilde{t})=u_h(\x)$ for all $u_h\in\Vh$ in this section.
We reformulate the weak formulation of~\eqref{eq:E_forward} by Green's identity with a test function $\psi_h$ defined on the $\Omega_0^h, \Gamma_H^h$, and $\Gamma^h$ as
\begin{equation*}  
0=\int_{\Omega_0^h}\sigma(\x)\nabla_{\mathbf{x}} v_h(\x)\cdot\nabla_{\mathbf{x}} \psi_h(\x)\mathrm{d}\x-\int_{\Gamma_H^h}\sigma(\x)\nabla_{\mathbf{x}} u_h(\x)\cdot \mathbf{n}\psi_h(\x)\mathrm{d}\x.
\end{equation*}
Then, we evaluate the weak formulation for all basis functions $\left(\varphi_{k}^{\Xcal}\right)_{k=1}^{N_{\Xcal}}$ on $\Xcal\in\left\{\Omega_0, \Gamma,\Gamma_H\right\}$ with $\varphi_{k}^{\Gamma_H}\coloneqq \varphi_{k}$ and $N_{\Gamma_H}\coloneqq N_{\Vcal}$ on $\Gamma_H$
\begin{align*}
&\sum_{l=1}^{N_{\Omega_0}}v_{l}\int_{\Omega_0^h}\sigma\nabla_{\mathbf{x}} \varphi_{l}^{\Omega_0}\cdot\nabla_{\mathbf{x}} \varphi_{k}^{\Xcal}\mathrm{d}\x+\sum_{j=1}^{N_\Gamma}\left(v|_\Gamma\right)_{j}\int_{\Gamma^h}\sigma\nabla_{\mathbf{x}} \varphi_j^{\Gamma}\cdot\nabla_{\mathbf{x}} \varphi_{k}^{\Xcal}\mathrm{d}\x\\
=&- \sum_{i=1}^{N_{\Gamma_H}}u_{i}\int_{\Gamma_H^h}\sigma\nabla_{\mathbf{x}} \varphi_i^{\Gamma_H}\cdot\nabla_{\mathbf{x}} \varphi_{k}^{\Xcal}\mathrm{d}\x.
\end{align*}
Rewriting the integrals with stiffness matrices results in
\begin{equation*}
    \begin{pmatrix}
        \A_{\Omega_0\Omega_0} & \A_{\Omega_0\Gamma}\\
         \A_{\Gamma\Omega_0} & \A_{\Gamma\Gamma}\\
        \end{pmatrix} 
        \begin{pmatrix}
         \mathbf{v} \\
          \mathbf{v}|_\Gamma \\
        \end{pmatrix}= -
    \begin{pmatrix}
        \A_{\Omega_0\Gamma_H}\\
         \A_{\Gamma\Gamma_H}
        \end{pmatrix}
        \mathbf{u}
\end{equation*}
with matrices
\begin{equation*}
\A_{\Xcal\mathcal{Y}} = \left(\int_{\Omega_0^h}\sigma(\x) \nabla_\mathbf{x}\varphi_i^{\Xcal}(\x)\nabla_\mathbf{x}\varphi_j^{\mathcal{Y}}(\x)\dx\x\right)_{1\leq i\leq N_\Xcal,1\leq j\leq N_\mathcal{Y}}
\end{equation*}
and $\Xcal,\mathcal{Y}\in\left\{\Omega_0,\Gamma,\Gamma_H\right\}$.
We explicitly compute the forward operator by taking advantage of $\A_{\Gamma\Gamma_H}=0$ and the invertibility of the stiffness matrix $\A_{\Omega_0\Omega_0}$.
Thus,
\begin{equation*}
\mathbf{v}|_{\Gamma} = \A^\Gamma\mathbf{u}, \quad \text{with} \quad\A^\Gamma = \left(\A_{\Gamma\Gamma}-\A_{\Gamma\Omega_0}\A_{\Omega_0\Omega_0}^{-1}\A_{\Omega_0\Gamma}\right)^{-1}\A_{\Gamma\Omega_0}\A_{\Omega_0\Omega_0}^{-1}\A_{\Omega_0\Gamma_H}.
\end{equation*}
The forward operator needs to be restricted on the set $\Sigma^h$ depicting the nodes on the torso where the potential is measured with electrodes such that $\A^\Sigma\subset\A^\Gamma$.
Thus, for $\mathbf{v}|_{\Sigma}\subseteq \mathbf{v}|_{\Gamma}$ we obtain $\mathbf{v}|_{\Sigma} = \A^\Sigma\mathbf{u}$.
Finally, we construct the block diagonal matrix for all timesteps $\llbracket S\rrbracket_0$ via
\begin{equation*}
\A=\mathbf{I}_{S+1}\tens{K}\A^\Sigma.
\end{equation*}

\section{Primal-Dual Optimization Algorithm}\label{sec:Primal-DualOptimization}
This section presents details on the primal-dual algorithms to solve the inverse problem in a finite element setting.

\subsection{Problem Formulation}
In what follows, we consider the subsequent convex finite element-based variational problem 
\begin{equation*}\label{eq:E_discretized}
    \min_{u_{h}\in \Vh} \frac{1}{2 |\Sigma^h|N_\Sigma}\int_{\Sigma^h\times\T^h}\left(A[u_{h}]-z_h\right)^2 \dx (\x,t)  +\int_{\Gamma_H^h\times\T^h}\TV_\epsilon^\alpha\left(K_h u_{h}\right)\dx (\x,t)
\end{equation*}
for $\alpha\in\{1,2\}$. 
Furthermore, we define the discretized functions of \eqref{eq:G_cont} and \eqref{eq:F_cont} as $G_h$ and $\Fa$, respectively, which are proper, convex, and lower semi-continuous.
Exploiting the properties of the convex conjugate \cite{Bo04}, we transform this problem into a saddle point via
\begin{equation}\label{eq:E_saddle}
    \min_{u\in\Vh} \max_{p_h\in\Qh^\alpha} G_h(u_h) + \left( K_h u_h, p_h \right)_{\Qh^{\alpha}} - (\Fa)^*(p_h),
\end{equation}
with $(\Fa)^*$ being the convex conjugate of $\Fa$, which is proper, convex, and lower semi-continuous.
Since these functions can be fixed to a finite element discretization and we only optimize the real-valued vectors of nodal and element values, we define $\mathbf{G}_h$ and $\mathbf{F}^\alpha_h$ as representations of $G_h$ and $F_h^\alpha$ such that
\begin{equation}
\mathbf{G}_h(\mathbf{u})\coloneqq G_h(u_h) \quad \text{and} \quad \mathbf{F}_h^\alpha(\mathbf{p})\coloneqq F_h^\alpha(p_h).
\end{equation}
Next, we compute the matrix-vector representation of the energy function. 
All computations for the total variation-based methods are computed using lumped mass matrices in space and time to improve computational time.
As a result, lumped mass matrices are applied in the computations in this section to simplify the problem at hand.
On the discretized set $\Sigma^h$, we assume the spatial weights to vanish since the size of the electrodes used to measure the BSPM can be chosen arbitrarily small.
Therefore, the function $\mathbf{G}_h$ can be expressed using matrix-vector notation as a time series without spatial weights
\begin{equation*}
\mathbf{G}_h(\mathbf{u})  = \frac{1}{2 N_\Sigma} \left(\A\uv-\mathbf{z}\right)^\top \widetilde{\D}^{\Sigma}\left(\A\uv-\mathbf{z}\right).
\end{equation*}
In the numerical computation of the regularizer, we assume $0<\epsilon\ll 1$ such that only the non-quadratic total variation regularization occurs to facilitate the presentation.
By assuming the boundedness of all involved functions and utilizing compactness arguments, it can be shown that there exists a mesh-dependent constant $\epsilon_0$ such that, for all $\epsilon<\epsilon_0$, only the non-quadratic components of the regularization are active.

For the gradient space $\Qho$, we compute the regularizer $\Fob(\p)$ by splitting the $L^1$ norm into the sum over separate integrals $\Fokb(\p)$ and $\Fodb(\p)$ for each spatial gradient dimension $k\in\llbracket d\rrbracket$ and the temporal gradient dimension, i.e.
\begin{align*}
\Fob(\p)&=\sum_{k=1}^d\Fokb(\mathbf{p})+\Fodb(\mathbf{p}) \quad \text{with} \quad\\
\Fokb(\p) &=\sum_{s=0}^S \sum_{l=1}^{N_{\Qcal}} d^{\Pf_1}_s\M^{\Pf_0}_{l,l}|p_{k,l,s}|,
\quad k\in\llbracket d \rrbracket,\\
\Fodb(\p) &= \sum_{j=1}^S \sum_{i=1}^{N_{\Vcal}}\D^0_{j,j}m_i^{\Pf_1}|p_{d+1,i,j}|
\end{align*}
Next, we compute the regularizer $\Ft$
\begin{align*}
&\Ftb(\p) =\sum_{j=1}^{S}\sum_{l=1}^{N_{\Qcal}}\sum_{i=1}^{d}\sum_{s=1}^{2}\frac{\D^0_{j,j}}{2}\frac{\M^{0}_{l,l}}{d}\norm{p_{i,l,s,j}}_{2}
\end{align*}
with the $\ell^2$ norm over the gradient dimension $\norm{p_{i,l,s,j}}_{2}=\sqrt{\sum_{k=1}^{d+1}p_{k,i,l,s,j}^2}$.
For optimization of the saddle point problem \eqref{eq:E_saddle}, we employ the first-order primal-dual algorithm \ref{alg:algorithm} (see \cite{Cha11}).
This method iteratively updates primal and dual variables according to the proximal mappings of $G_h$ and $(\Fa)^*$ starting from an initial guess.
Recall that the proximal mapping of  $H:\mathcal{X} \to \mathcal{Y}$ is defined as 
\begin{equation*}
\prox_{H} (\widetilde{\mathbf{x}}) = \argmin_{\mathbf{x}\in\mathcal{X}} \frac{1}{2} \norm{\mathbf{x} - \widetilde{\mathbf{x}}}_{\mathcal{X}}^2 + H(\mathbf{x}).
\end{equation*}
\begin{algorithm}[tb]
\caption{First-Order Primal-Dual Algorithm}\label{alg:algorithm}
\begin{enumerate}
    \item \textbf{Initialization:} Choose $\tau, \sigma > 0$ with $\tau\sigma\norm{K_h}^2_{(\Vh, \Qh^{\alpha})}\leq 1$, $\theta \in [0, 1]$, $\textbf{u}^0\sim \mathcal{N}(0,1)$, $\p^0=\K^{\alpha}\textbf{u}^0$, and set $\bar{\textbf{u}}^0 = \textbf{u}^0$.
    \item \textbf{Iterations} ($n \geq 0$): Update $\textbf{u}^n, \textbf{p}^n, \bar{\textbf{u}}^n$ as follows:
    \[
    \begin{cases}
        \textbf{p}^{n+1} = \prox_{\sigma (\Fab)^*}(\textbf{p}^n + \sigma \K^{\alpha} \bar{\textbf{u}}^n) \\
        \textbf{u}^{n+1} = \prox_{\tau \mathbf{G}_h}(\textbf{u}^n - \tau (\K^{\alpha})^*\textbf{p}^{n+1}) \\
        \bar{\textbf{u}}^{n+1} = \textbf{u}^{n+1} + \theta(\textbf{u}^{n+1} - \textbf{u}^n)
    \end{cases}
    \]
\end{enumerate}
\end{algorithm}
Subsequently, we compute the proximal mappings for the optimization algorithm: 
\begin{equation*}
\prox_{\tau \mathbf{G}_h} (\widetilde{\mathbf{u}}) =\argmin_{\mathbf{u}\in\R^{N_{\Vcal} (S+1)}}\frac{1}{2}(\mathbf{u}-\widetilde{\mathbf{u}})^\top\widetilde{\D}^\Vcal\widetilde{\M}^\Vcal(\mathbf{u}-\widetilde{\mathbf{u}})+\frac{\tau}{2 N_\Sigma} \left(\A\uv-\mathbf{z}\right)^\top \widetilde{\D}^{\Sigma}  (\A\uv-\mathbf{z}).
\end{equation*}
Differentiating with respect to $\mathbf{u}$ and rearranging the optimality condition results in 
\begin{equation*}
\mathbf{u} =\left(\frac{\tau}{N_\Sigma} (\A^\top\widetilde{\D}^{\Sigma}\A)+\widetilde{\D}^\Vcal\widetilde{\M}^\Vcal\right)^{-1}\left(\frac{\tau}{N_\Sigma} \A^\top \widetilde{\D}^{\Sigma} \mathbf{z}+ \widetilde{\D}^\Vcal\widetilde{\M}^\Vcal\widetilde{\mathbf{u}}\right).
\end{equation*}
As a consequence of the mixed-product property of the Kronecker product $\tens{K}$ and matrix multiplication, the temporal weights vanish since $\A^\top\widetilde{\D}^{\Sigma}=\widetilde{\D}^{\Vcal}\A^\top$, which readily implies 
\begin{equation*}
\prox_{\tau \mathbf{G}_h}(\widetilde{\mathbf{u}})  = \left(\frac{\tau}{N_\Sigma} (\A^\top\A)+\widetilde{\M}^\Vcal\right)^{-1}\left(\frac{\tau}{N_\Sigma} \A^\top \mathbf{z}+ \widetilde{\M}^\Vcal\widetilde{\mathbf{u}}\right).
\end{equation*}
For the proximal mappings of $(\Fob)^*$ and $(\Ftb)^*$ we initially compute the convex conjugate functions starting with $(\Fokb)^*(p_{h})$ for $k\in\llbracket d\rrbracket$ and $(\Fodb)^*(p_{h})$
\begin{equation*}
(\Fokb)^*(\p) =\delta_{C^1_k}(\mathbf{p}_{k}),\quad \text{for}\quad k\in\llbracket d \rrbracket,\quad\text{and}\quad(\Fodb)^*(\p) =\delta_{C^1_{d+1}}(\mathbf{p}_{d+1}),
\end{equation*}
for the sets
\begin{align*}
C^1_k&\coloneqq \left\{ \mathbf{p}_k:\frac{\widetilde{\D}^{\Pf_{1}}_{s,s}}{d^{\Pf_{1}}_s}|p_{k,l,s}| <1\quad\forall (l,s)\in\llbracket N_\Qcal\rrbracket\times\llbracket S\rrbracket_0 \right\}, \quad k\in\llbracket d \rrbracket,\\
C^1_{d+1}&\coloneqq \left\{ \mathbf{p}_{d+1}:\frac{\widetilde{\M}^{\Pf_1}_{i,i}}{m^{\Pf_1}_i}|p_{d+1,i,j}|<1\quad\forall (i,j)\in\llbracket N_\Vcal\rrbracket\times\llbracket S\rrbracket \right\},
\end{align*} 
as well as the convex conjugate of $(\Ftb)^*$
\begin{equation*}
(\Ftb)^*(\p) =
\delta_{C^2}(\mathbf{p}),
\end{equation*}
for the set
\begin{equation*}
C^2\coloneqq \{\mathbf{p}:\norm{p_{i,l,s,j}}_2<1\quad\forall (i,l,s,j)\in\llbracket d\rrbracket\times\llbracket N_{\Qcal}\rrbracket\times\llbracket 2\rrbracket\times\llbracket S\rrbracket \}.
\end{equation*} 
The resulting proximal mappings for the gradient space $\Qho$ are evaluated separately for each element vector component
\begin{equation*}
\left(\prox_{\sigma (\Fokb)^*} (\widetilde{\p})\right)_{k,l,s} = \widetilde{p}_{k,l,s}/\max\left(1,\frac{\widetilde{\D}^{\Pf_1}_{s,s}|\widetilde{p}_{k,l,s}|}{d^{\Pf_1}_s}\right)
\end{equation*}
for $(k,l,s)\in\llbracket d\rrbracket\times\llbracket N_{\Qcal}\rrbracket\times\llbracket S\rrbracket_0$.
Likewise, the remaining proximal mapping reads as
\begin{equation*}
\left(\prox_{\sigma (\Fodb)^*} (\widetilde{\p})\right)_{d+1,i,j} =\widetilde{p}_{d+1,i,j}/\max\left(1,\frac{\widetilde{\M}^{\Pf_1}_{i,i}|\widetilde{p}_{d+1,i,j}|}{m^{\Pf_1}_i}\right)
\end{equation*}
for $(i,j)\in\llbracket N_{\Vcal}\rrbracket\times\llbracket S\rrbracket$.
Finally, the proximal mapping of $(\Ftb)^*$ is also computed vector element-wise as
\begin{equation*}
\left(\prox_{\sigma (\Ftb)^*} (\widetilde{\p})\right)_{k,i,l,s,j} 
=\widetilde{p}_{k,i,l,s,j}/\max\left(1,\norm{\widetilde{p}_{i,l,s,j}}_2\right)
\end{equation*}
for $(k,i,l,s,j)\in\llbracket d+1 \rrbracket\times\llbracket d\rrbracket\times\llbracket N_{\Qcal}\rrbracket\times\llbracket 2\rrbracket\times\llbracket S\rrbracket$.
The primal-dual algorithm \ref{alg:algorithm} requires the Lipschitz constant of the weighted total variation operator $K_h$ as it is essential for the convergence
\begin{equation*}   \norm{K_h}^2_{\left(\Vh,\Qh^{\alpha}\right)} = \sup_{u_h\in\Vh}\frac{\norm{K_h u_h}_{\Qh^{\alpha}}^2}{\norm{u_h}_{\Vh}^2} = \lambda_{\mathrm{max}}((\K^{\alpha})^*\K^{\alpha}).
\end{equation*}

\section{Numerical Results} \label{sec:numerical_results}
This section introduces different state-of-the-art finite element regularization methods for solving the inverse problem. 
Moreover, we describe the simulation of electrical activity on the epicardium, which serves as the ground truth function, and visualize the two- and three-dimensional reconstruction results for both the baseline methods and our novel total variation-based approach in the presence of noise.

\subsection{Baseline Methods}
The most common regularization approach applied in the inverse problem in electrocardiographic imaging is Tikhonov regularization~\cite{Cl15, Ti77}.
For the energy minimization problem incorporating Tikhonov regularization, we adopt the following model
\begin{equation}\label{eq:E_tikhonov}
    \min_{u_{h}\in \Vh} G_h(u_h)  +\frac{\lambda_\gamma}{2}\norm{L_h u_{h}}_{\Vh}^2 
\end{equation}
with $L_h$ being the identity $I_h$ for zero-order Tikhonov ($\Tz$) or the spatial gradient $\nabla_{\x}$ for first-order Tikhonov in space ($\To$).
Since the regularizer function of Tikhonov is differentiable, we compute the explicit solution for $\Tz$ and $\To$ by rearranging the optimality condition of \eqref{eq:E_tikhonov}.
For zero-order Tikhonov with $L_h = I_h$ differentiating with respect to $\mathbf{u}$ results in
\begin{equation*}
    0 = \frac{1}{ N_\Sigma}(\A^\top \D^\Sigma\A\mathbf{u}-\A^\top \D^\Vcal\mathbf{z}) + \lambda_\gamma\D^\Vcal\M^\Vcal \mathbf{u}.
\end{equation*}
The weights of the mass matrices $\D^\Sigma$ and $\D^\Vcal$ vanish after solving for $\mathbf{u}$ since $\D^\Sigma \A = \A \D^\Vcal$ and $\D^\Vcal\M^\Vcal = \M^\Vcal\D^\Vcal$.
Due to the remaining matrices $\A$ and $\M^\Vcal$ being timestep-wise uniform, we can solve the linear system independently at each timestep
\begin{equation*}
    \mathbf{u}(t_s) = (\frac{1}{ N_\Sigma}(\A^\Sigma)^\top\A^\Sigma + \lambda_\gamma \M^{\Pf_1})^{-1}\frac{1}{ N_\Sigma}(\A^\Sigma)^\top \mathbf{z}(t_s)
\end{equation*}
for each $t_s$ with $\mathbf{u}(t_s)\coloneqq (u_{1,s},\ldots,u_{N_\Vcal,s})^\top$ and $s\in\llbracket S\rrbracket_0$.
Following the same steps for $\To$, we compute
\begin{equation*}
    \mathbf{u}(t_s) = (\frac{1}{ N_\Sigma}(\A^\Sigma)^\top\A^\Sigma - \lambda_\gamma \Delta_{\x,h})^{-1}\frac{1}{ N_\Sigma}(\A^\Sigma)^\top \mathbf{z}(t_s),
\end{equation*}
with the discretized spatial Laplacian $\Delta_{\x,h}$.
As introduced in \cite{Os92} for BEM methods, we propose a method with a first-order Tikhonov term and additionally compare the reconstruction with an estimate of the epicardial potential $v_h\in\Vh$ which is chosen as the potential of the previous timestep. 
Since we need to solve this regularization iteratively, the optimization problem norms are solely spatial.
Let us define $\Vh^{t_s}$ as the space of all functions $u_h\in\Vh$ evaluated at time $t_s$.
With $\mathbf{u}(t_{-1})=(0,\ldots,0)^\top\in\R^{N_{\Vcal}}$ ensuring Dirichlet boundary conditions, the first-order Tikhonov method with additional time regularization ($\TTo$) is determined by minimizing for $u_h(t_s)\in \Vh^{t_s}$ time stepwise
\begin{equation*}\label{eq:E_tikhonov_time}
    \min_{u_{h}(t_s)} \frac{1}{2 N_{\Sigma}}\norm{(A[u_h]-z_h)(t_s)}_{L^2(\Sigma^h)}^2 +\frac{\lambda_\gamma}{2}\norm{\nabla_\x u_{h}(t_s)}_{\Vh^{t_s}}^2+\frac{\lambda_t}{2}\norm{u_h(t_{s})-u_h(t_{s-1})}_{\Vh^{t_s}}^2
\end{equation*}
for all $s\in\llbracket S\rrbracket_0$.
Solving the optimality condition for each timestep results in computing
\begin{equation*}
    \mathbf{u}(t_s) = (\frac{1}{ N_\Sigma}(\A^\Sigma)^\top\A^\Sigma - \lambda_\gamma \Delta_{\x,h} + \lambda_t \M^{\Pf_1})^{-1}(\frac{1}{ N_\Sigma}(\A^\Sigma)^\top \mathbf{z}(t_s)+\lambda_t\M^{\Pf_1}\mathbf{u}(t_{s-1})).
\end{equation*}

\subsection{Ground Truth Generation}
This section shows the approach for simulating a potential on a two-dimensional model (see \Cref{fig:2D_model}) containing lungs and $16$ torso electrodes~\cite{Gander2021UQ} and the three-dimensional rabbit torso-heart model with $32$ electrodes \cite{Mo22}.
The 2D model setup considers a simple cylindrical torso with the electrodes distributed around the torso surface. 
To simulate cardiac activation times $\phi$, we used an anisotropic eikonal model with several initial points~\cite{Fr14}.
The transmembrane potential $v_m$ can be computed from the eikonal solution by means of a fixed waveform~\cite{Pe17}:
\begin{equation*}
    v_m(\x,t) = R_0 + \frac{R_1 - R_0}{2} \left[ \tanh \left( \frac{2 (t - \phi(\x)) }{\kappa}\right) + 1 \right].
    \label{eq:tanh_waveform}
\end{equation*}
Note that for the depolarization and repolarization constants, we chose $R_1 = 85$ mV, $R_0 = -30$ mV, and for the time-scale $\kappa = 1$.
The model consists of $4$ different regions, each containing different conductivities $\sigma_{i/e}$ in $S/m$ in line with other studies~\cite{Ke10,Mo22}.
The chosen conductivities in 2D are shortly summarized in \Cref{tab:simulation_conductivities}.
The 3D rabbit torso model shown in \Cref{fig:illustration} represents the upper body without the head but includes the arms and several organs.
For further information and the conductivities for each organ of the 3D rabbit model, we refer the reader to~\cite[Tab.~1]{Mo22}.
\begin{table}[htb]
\centering
\begin{scriptsize}
    \begin{tabular}{c|c|c}
        Region & $\sigma_i$ & $\sigma_e$ \\ \hline
        Torso & 0 & 0.22 \\
        Lungs & 0 & 0.03 \\
        Blood & 0 & 0.7 \\
        Myocardium $(\Omega_H)$ & $\left( 0.174, 0.0193 \right)$ & $\left( 0.625, 0.236 \right)$ \\
        \hline
    \end{tabular}
\end{scriptsize}
\caption{Conductivities used for simulating the ground truth activation. 
Multiple values indicate anisotropic values along the fiber and transversal directions. 
All units in $S/m$.}\label{tab:simulation_conductivities}
\end{table}

To simulate the ground truth activation, we utilize the so-called pseudo bidomain equation~\cite{Bi11}, a computationally simpler version of the bidomain model~\cite{Li03}.
In the absence of any boundary conditions, we additionally add a small $\varepsilon > 0$ to guarantee well-posedness:
\begin{equation}
        - \operatorname{div} \left(\sigma(\x) \nabla v(\x, t) \right) + \varepsilon v(\x, t) = \operatorname{div} \left( \sigma_i(\x) v_m (\x, t) \right) \quad (\x, t) \in \Omega \times \T,
        \label{eq:pseudo_bidomain}
\end{equation}
where $\sigma_{i/e}$ are the intra- and extracellular conductivities respectively (note that $\sigma_i = \mathbf{0}$ for all $\x \notin \Omega_H$ and $\sigma = \sigma_i + \sigma_e$ as in~\eqref{eq:E_forward}). 
Here, $v$ and $v_m$ are the extracellular and transmembrane potentials, respectively.

We consider a fixed geometry for both models and all numerical simulations are computed using a finite element method with $\Pf_1$ elements.
The 2D model contains $17,292$ triangle elements and $8,728$ vertices with $210$ nodes and $210$ two-dimensional elements representing the epicardium.
The 3D rabbit dataset consists of $475,574$ tetrahedral elements and $82,997$ vertices with $20,590$ vertices and $41,184$ cells on the epicardium.

\subsection{Results}
The novel space-time total variation-based methods for the regularizer $\Fa$ are subsequently defined as $\TVSTa$.
To highlight the benefit of time regularization, we also compare all results to the spatial total variation $\TVSa$ induced by the regularizer $\Fa$ with $\lambda_t=0$.
All total variation-based methods are computed with lumped mass matrices in contrast to the baseline methods.
To compare the results obtained by the different approaches, we define the frequently used \emph{relative error} (RE) and \emph{Pearson's correlation coefficient} (CC) evaluating the vector of nodal values $\mathbf{u}$ of a discretized finite element function $u_h\in\Vh$, as well as the \emph{$\Vh$ error} for discretized functions taking the geometry of the problem in consideration:
\begin{equation*}
\text{RE}(\mathbf{u}) = \frac{\norm{\mathbf{u}-\mathbf{u}^\mathrm{g}}_{2}}{\norm{\mathbf{u}^\mathrm{g}}_{2}}, \quad
\text{CC}(\mathbf{u}) = \frac{(\mathbf{u}-\overline{\mathbf{u}},\mathbf{u}^\mathrm{g}-\overline{\mathbf{u}}^\mathrm{g})_{2}}{\norm{\mathbf{u}-\overline{\mathbf{u}}}_{2}\norm{\mathbf{u}^\mathrm{g}-\overline{\mathbf{u}}^\mathrm{g}}_{2}},\quad
\Vh^{\mathrm{err}}(u_h)=\norm{u_h-u_h^\mathrm{g}}_{\Vh},
\end{equation*}
where $\mathbf{u}^\mathrm{g}$ represents the vector of nodal values for the ground truth function $u_h^\mathrm{g}$, and $\overline{\mathbf{u}}$ is defined as the mean of $\mathbf{u}\in\R^{N_\Vcal(S+1)}$.
The relative error gain of our method, denoted as $\epsilon^{\text{our}}$, compared to the baseline method, denoted as $\epsilon^{\text{base}}$, is calculated using the formula $\frac{\lvert \epsilon^{\text{our}} - \epsilon^{\text{base}} \rvert}{\epsilon^{\text{base}}}$.
The introduced baseline and total variation-based methods are evaluated in two- and three-dimensional applications.
We remark that solving the inverse problem is performed using the same discretization as in the simulation of the ground truth, which may lead to more optimistic results compared to real-world scenarios.
All computations are performed using CuPy \cite{CUPY} and SciPy \cite{SCIPY} on an NVIDIA A40 GPU and an AMD EPYC 7543 CPU.
All computations are performed in double precision.
The algorithm is terminated if the  $l^\infty$-norm difference of two consecutive primal iterates, i.e., $\norm{\mathbf{u}^{n+1}-\mathbf{u}^{n}}_{\infty}$, is below the threshold value $10^{-3}$.
A logarithmic grid search minimizing the $\Vh$ error has been conducted for all methods to determine the optimal hyperparameter values in the range between $10^{-15}$ and $1$.
Although the total variation-based methods allow for different spatial and temporal regularization penalization schemes, the applications discussed in this paper induce equal space and time weighting in a logarithmic grid search for the unrelated physical dimensions, implying correlation.

Computing the reconstruction of the simulated body surface potential without noise perturbations induces relatively small optimal regularization parameters for all approaches.
To investigate the robustness and ensure numerical stability, we compute reconstructions with component-wise additive white Gaussian noise $\mathbf{n}\sim \mathcal{N}(0,\sigma^2_\mathbf{n})$ on the vector of nodal values of the body surface measurements $\mathbf{z}$.
The noise is applied for different \textit{Signal to Noise Ratio} levels in ${\db}$ as it is common in literature \cite{Bc17}.
For the simulated body surface potential $z_h^{\mathrm{g}}\coloneqq A[u_h^g]$ with vector of nodal values $\mathbf{z}^{\mathrm{g}}$, we compute the noisy measurement function by $\mathbf{z} =\mathbf{z}^{\mathrm{g}} + \mathbf{n}$ with noise level
\begin{equation*}
\mathrm{SNR}_{\db}=20 \log_{10}\frac{\norm{\mathbf{z}}_{2}}{\norm{\mathbf{n}}_{2}}.
\end{equation*}

\subsubsection{Two-Dimensional Model}
Our first test is to reconstruct simulated two-dimensional measurements on the torso for noise levels $70$, $50$, and $20 \db$.
Numerical errors of all methods proposed in this work are evaluated in \Cref{tab:2D_results}.
The space-time reconstructions for $50 \db$ noise are compared to the ground truth simulation in a two-dimensional plot in \Cref{fig:2D_noise_50} with the x-axis being the angle of the circular heart and the y-axis being the temporal dimension.
A comparison of the simulated potential and different reconstructions of the extracellular potential at a randomly chosen vertex of the epicardium can be observed in \Cref{fig:2D_noise_ecg}.
$\To$ and $\TTo$ as well as the total variation-based $\TVSTo$ and $\TVSTt$ are also visualized for $70$ and $20\db$ noise in \Cref{fig:2D_noise_70_20}.
\begin{figure}[htp] \centering{
\begin{tikzpicture}
    \node[inner sep=0pt] (A) at (0,0){\includegraphics[clip, trim=0.2cm 0.1cm 0cm 0.1cm,width=0.98\textwidth]{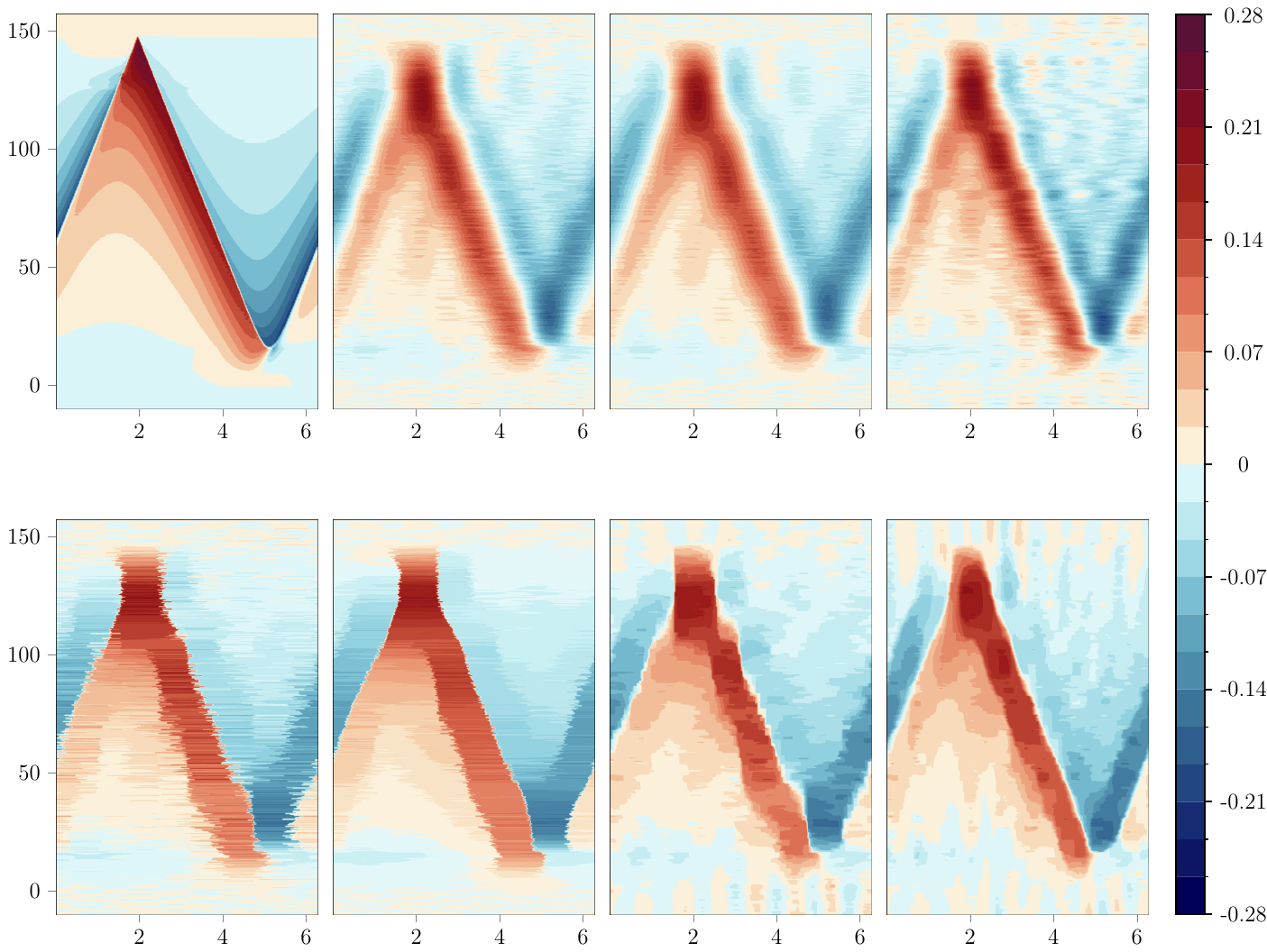}};
    \node[inner sep=0pt, rotate=90] at (-6.5,2.65) 
    {{\tiny Time $t$ [ms]}};
    \node[inner sep=0pt, rotate=90] at (-6.5,-2.5) 
    {{\tiny Time $t$ [ms]}};

    \foreach \x/\y in {0/$\GT$,1/$\Tz$,2/$\To$,3/$\TTo$}
    {\node[inner sep=0pt] at (-4.65+\x*2.8,4.95) 
    {\y};}

    \foreach \x/\y in {0/$\TVSo$,1/$\TVSt$,2/$\TVSTo$,3/$\TVSTt$}
    {\node[inner sep=0pt] at (-4.65+\x*2.8,4.95-5.1) 
    {\y};}

    \foreach \x in {0,...,3}
    {\foreach \y in {0,1}
    {\node[inner sep=0pt] at (-4.65+\x*2.8,0.25-\y*5.15){{\tiny Angle [rad]}};}}
\end{tikzpicture}}
\caption{Two-dimensional heart potential reconstruction displayed as a function over space (Angle [rad]) and time (Time t [ms]) with ground truth function ($\GT$) for $50\db$ noise.}
\label{fig:2D_noise_50}
\end{figure}
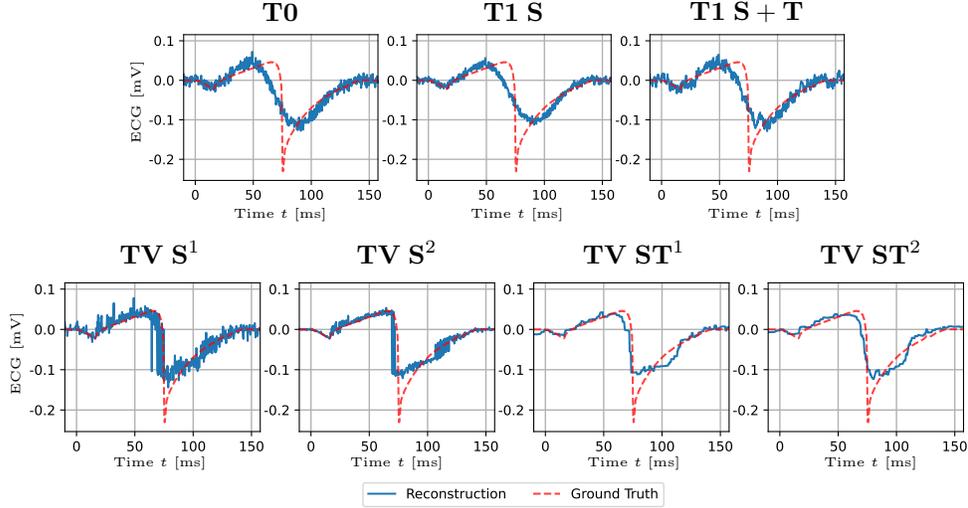
\begin{figure}[htp] \centering{
\begin{tikzpicture}
    \node[inner sep=0pt] (A) at (0,0){\includegraphics[clip, trim=0.2cm 0.2cm 0.2cm 0.3cm,width=0.98\textwidth]{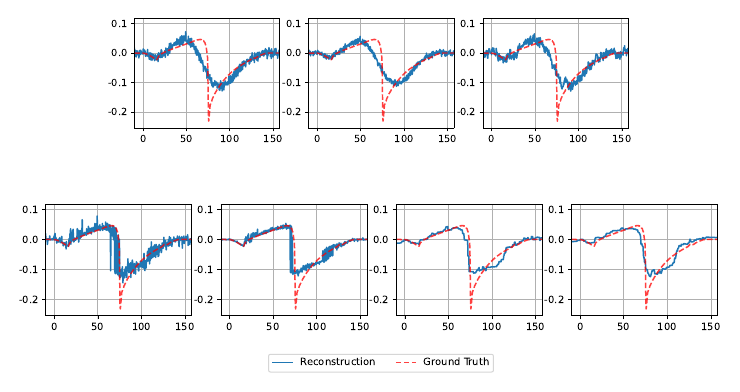}};
\node[inner sep=0pt, rotate=90] at (-4.8,2.3) 
    {{\tiny ECG [mV]}};
    \node[inner sep=0pt, rotate=90] at (-6.4,-1.05) 
    {{\tiny ECG [mV]}};

    \foreach \x/\y in {0/$\Tz$,1/$\To$,2/$\TTo$}
    {\node[inner sep=0pt] at (-2.9+\x*3.1,3.5) 
    {\y};}

    \foreach \x/\y in {0/$\TVSo$,1/$\TVSt$,2/$\TVSTo$,3/$\TVSTt$}
    {\node[inner sep=0pt] at (-4.5+\x*3.15,3.5-3.2) 
    {\y};}

    \foreach \x in {0,...,2}
    {\node[inner sep=0pt] at (-2.9+\x*3.15,0.8){{\tiny Time $t$ [ms]}};}
    
    \foreach \x in {0,...,3}
    {\node[inner sep=0pt] at (-4.5+\x*3.15,0.8-3.3){{\tiny Time $t$ [ms]}};}
\end{tikzpicture}}
\caption{Two-dimensional heart potential (ECG [mV]) with $50\db$ noise reconstructed by different methods over time (Time t [ms]) evaluated on an epicardial node of the corresponding finite element discretization.}
\label{fig:2D_noise_ecg}
\end{figure}
\begin{table}[htbp]
    \centering
    \begin{scriptsize}
    \label{tab:2D_results}
    \begin{tabular}{c|c|c|c|c|c|c|c|c}
     $\db$ & \diagbox[width=5.6em]{Error}{Reg.} & $\Tz$ & $\To$ & $\TTo$ & $\TVSo$& $\TVSt$& $\TVSTo$& $\TVSTt$\\
        \hline
          \multirow{4}*{70}& $\Vh \; \downarrow$ & $5.435$ & $5.454$ & $5.181$ & $5.42$ & $5.005$ & $\mathit{4.542}$ & $\mathbf{4.017}$\\
          & RE $\downarrow$ & $0.442$ & $0.443$ & $0.421$ & $0.441$ & $0.407$ & $\mathit{0.369}$ & $\mathbf{0.327}$\\
         & CC $\uparrow$ & $0.897$ & $0.896$ & $0.907$ & $0.899$ & $0.914$ & $\mathit{0.929}$ & $\mathbf{0.945}$\\
         \cline{2-9}
         & $\lambda_\gamma,\lambda_t$ & 1e-10 & 1e-8 & 1e-9,1e-9 & 1e-10& 1e-10 & 1e-10,1e-10 & 1e-11,1e-11\\
        \hline 
         \multirow{4}*{50} & $\Vh \; \downarrow$ & $6.042$ & $6.094$ & $5.695$ & $6.447$ & $5.916$ & $\mathit{5.325}$ & $\mathbf{4.454}$\\
         & RE $\downarrow$ & $0.491$ & $0.495$ & $0.463$ & $0.524$ & $0.481$ & $\mathit{0.433}$ & $\mathbf{0.362}$\\
         & CC $\uparrow$ & $0.871$ & $0.869$ & $0.886$ & $0.854$ & $0.877$ & $\mathit{0.901}$ & $\mathbf{0.932}$\\
         \cline{2-9}
         & $\lambda_\gamma,\lambda_t$ & 1e-8 & 1e-6 & 1e-8,1e-8 & 1e-8& 1e-8 & 1e-9,1e-9 & 1e-10,1e-10\\
        \hline 
        \multirow{4}*{20} & $\Vh \; \downarrow$ & $7.905$ & $7.827$ & $7.134$ & $8.854$ & $8.142$ & $\mathit{7.039}$ & $\mathbf{6.238}$\\
         & RE $\downarrow$ & $0.642$ & $0.636$ & $0.58$ & $0.72$ & $0.662$ & $\mathit{0.572}$ & $\mathbf{0.507}$\\
         & CC $\uparrow$ & $0.768$ & $0.772$ & $0.815$ & $0.702$ & $0.75$ & $\mathit{0.82}$ & $\mathbf{0.862}$\\
         \cline{2-9}
         & $\lambda_\gamma,\lambda_t$ & 1e-5 & 1e-3 & 1e-5,1e-5 & 1e-5& 1e-5 & 1e-6,1e-6 & 1e-7,1e-7\\
    \end{tabular}
        \end{scriptsize}
    \caption{$\Vh$ error, relative error (RE), correlation coefficient (CC), and regularization parameters $(\lambda_\gamma,\lambda_t)$ for two-dimensional heart potential reconstruction of different regularization methods with respect to the ground truth function.}
\end{table}

All the regularizing methods reconstruct a similar pattern to the ground truth function for the almost noiseless scenario ($70 \db$).
From a numerical perspective, we observe that $\To$ generates the worst reconstructions for the two-dimensional simulation, but adding prior information of the previous timestep and computing $\TTo$ improves the results and outperforms $\Tz$. 
With increasing noise levels, the baseline Tikhonov methods do not differ visually.
For SNR levels $50\db$ and $20\db$, $\TVSa$ improves the result both qualitatively in the space-time domain, as well as numerically for all applied $\alpha$ since the plot does not appear to be as blurred and contains sharper edges but is still perturbed by noise.
Numerically and visually, the $\TVSt$ method consistently generates superior results to $\TVSo$.
In summary, the best results can be achieved by $\TVSTa$ because the smoothing in time enhances the transitions even more.
Especially, $L^{2,1}$ regularization smoothes irregularities in the spikes of the space-time plot to sharper and more accurate reconstructions.
The benefit of adding time regularization for total variation is particularly visible in the extracellular plot \Cref{fig:2D_noise_ecg}.
However, for increasing regularization parameters $\lambda$, needed for higher noise levels the space-time total variation reconstruction exhibits smoothing effects at the spikes of the function.
The small regularization parameters for small noise levels indicate the highly ill-posedness of the problem for which little smoothing of the function is needed, but a low amount of regularization applied to the result leads to the desired properties of the potential on the heart.
Comparing the best-performing $\TVSTt$ method for $50\db$ noise to the best-performing Tikhonov with additional time regularization $\TTo$, the relative improvement in error values is $21.8\%$ in the $\Vh$ error, $21.8\%$ for the RE, and $5.2\%$ for the CC.
The computing time amounts to $0.0018$s for $\Tz$, $0.0023$s for $\To$, $0.0593$s for $\TTo$, at most $4.23$s for $\TVSo$ and $\TVSTo$, and at most $6.15$s for $\TVSt$ and $\TVSTt$.
\begin{figure}[htp] \centering{
\begin{tikzpicture}
\node[inner sep=0pt] (A) at (0,0){\includegraphics[clip, trim=0.2cm 0.1cm 0cm 0.1cm,width=0.93\textwidth]{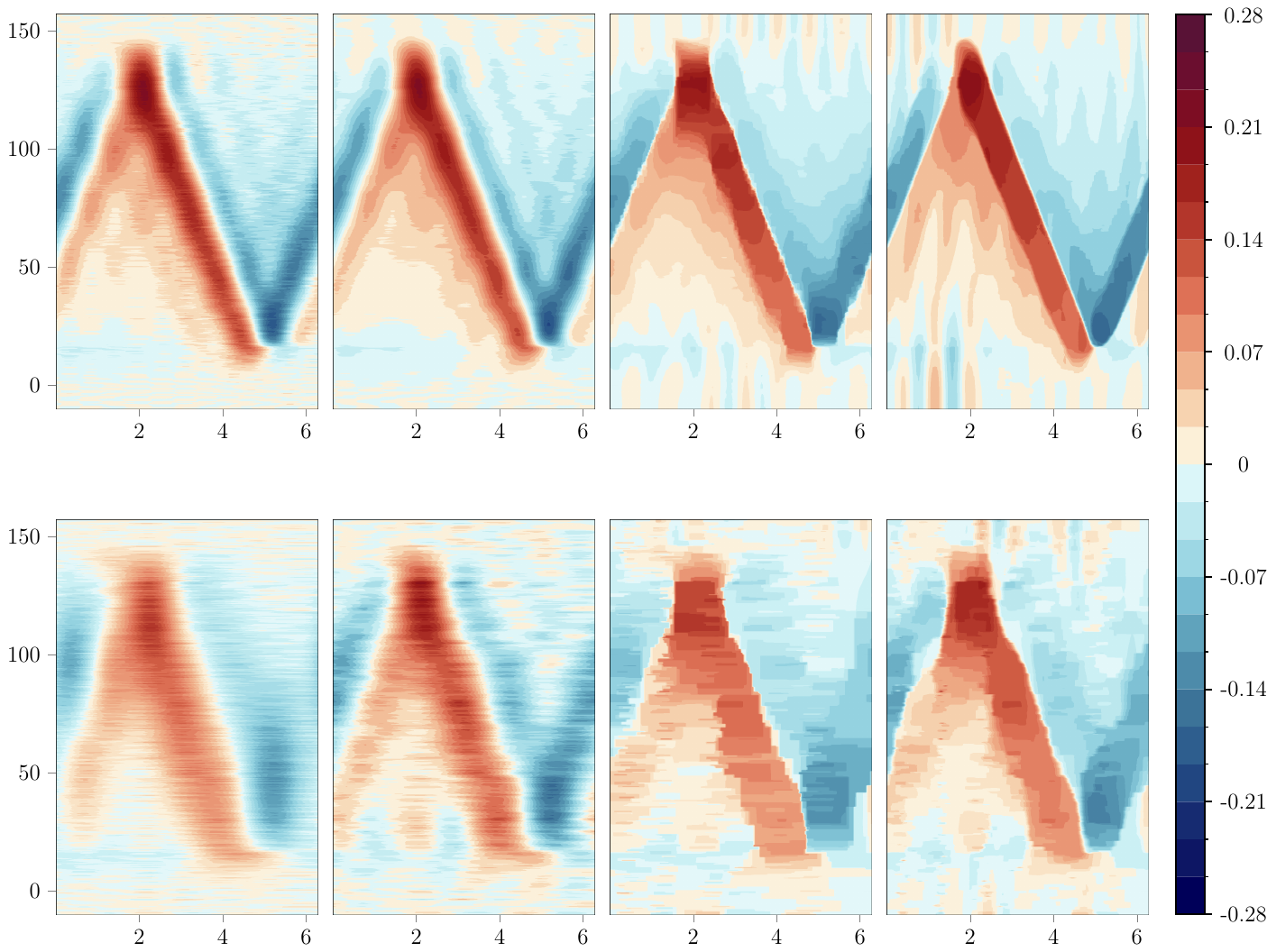}};
    \node[inner sep=0pt, rotate=90] at (-6.2,2.65) 
    {{\tiny Time $t$ [ms]}};
    \node[inner sep=0pt, rotate=90] at (-6.2,-2.25) 
    {{\tiny Time $t$ [ms]}};

    \foreach \x/\y in {0/$\Tz$,1/$\TTo$,2/$\TVSTo$,3/$\TVSTt$}
    {\node[inner sep=0pt] at (-4.5+\x*2.68,4.75) 
    {\y};}

    \foreach \x in {0,...,3}
    {\foreach \y in {0,1}
    {\node[inner sep=0pt] at (-4.5+\x*2.68,0.25-\y*4.9){{\tiny Angle [rad]}};}}
    \foreach \a/\b in {0/70,1/20}{
    \node[inner sep=0pt, rotate=90] at (-6.8,2.65-4.9*\a) 
        {$\b \db$};}
    \draw[line width=0.5mm] (-6.2,-0.1) -- (5,-0.1);
\end{tikzpicture}}
\caption{Two-dimensional heart potential reconstruction displayed as a function over space (Angle [rad]) and time (Time t [ms]) with ground truth function ($\GT$) for $70$ and $20\db$ noise.}
\label{fig:2D_noise_70_20}
\end{figure}

Since all computations are based on simulated observations and a ground truth function, the regularization parameters can be selected optimally. 
To justify the selection of these parameters, we conduct a study where the observation data is randomly split into subsets of $\{6.7\%, 20\%, 33.3\%\}$ corresponding to observations $\widetilde{N}_\Sigma\in\{1, 3, 5\}$ for generalized cross-validation (GCV)~\cite{Ba16,Go79}, with the remaining data used to compute the reconstruction of the inverse problem for $\TVSTt$ to determine the regularization parameters.
The validation error of the reconstruction by a limited number of observations $\mathbf{u}^{\widetilde{N}_\Sigma}$ is derived by the discretized data fidelity term $\mathbf{G}$
\begin{equation*}
\frac{1}{2 \widetilde{N}_\Sigma} \left(\A\mathbf{u}^{\widetilde{N}_\Sigma}-\mathbf{z}\right)^\top \widetilde{\D}^{\Sigma}\left(\A\mathbf{u}^{\widetilde{N}_\Sigma}-\mathbf{z}\right).
\end{equation*}
The optimally chosen parameters are uniform in both space and time, resulting in the GCV selection being also performed uniformly across both dimensions.
Evaluation of the generalized cross-validation in 2D with $50$dB noise is visualized in \Cref{fig:gcv}.
\begin{figure}[htp] \centering{
\begin{tikzpicture}
\node[inner sep=0pt] (Z) at (0,0){\includegraphics[clip, trim=0.1cm 0cm 0cm 0cm,width=0.93\textwidth]{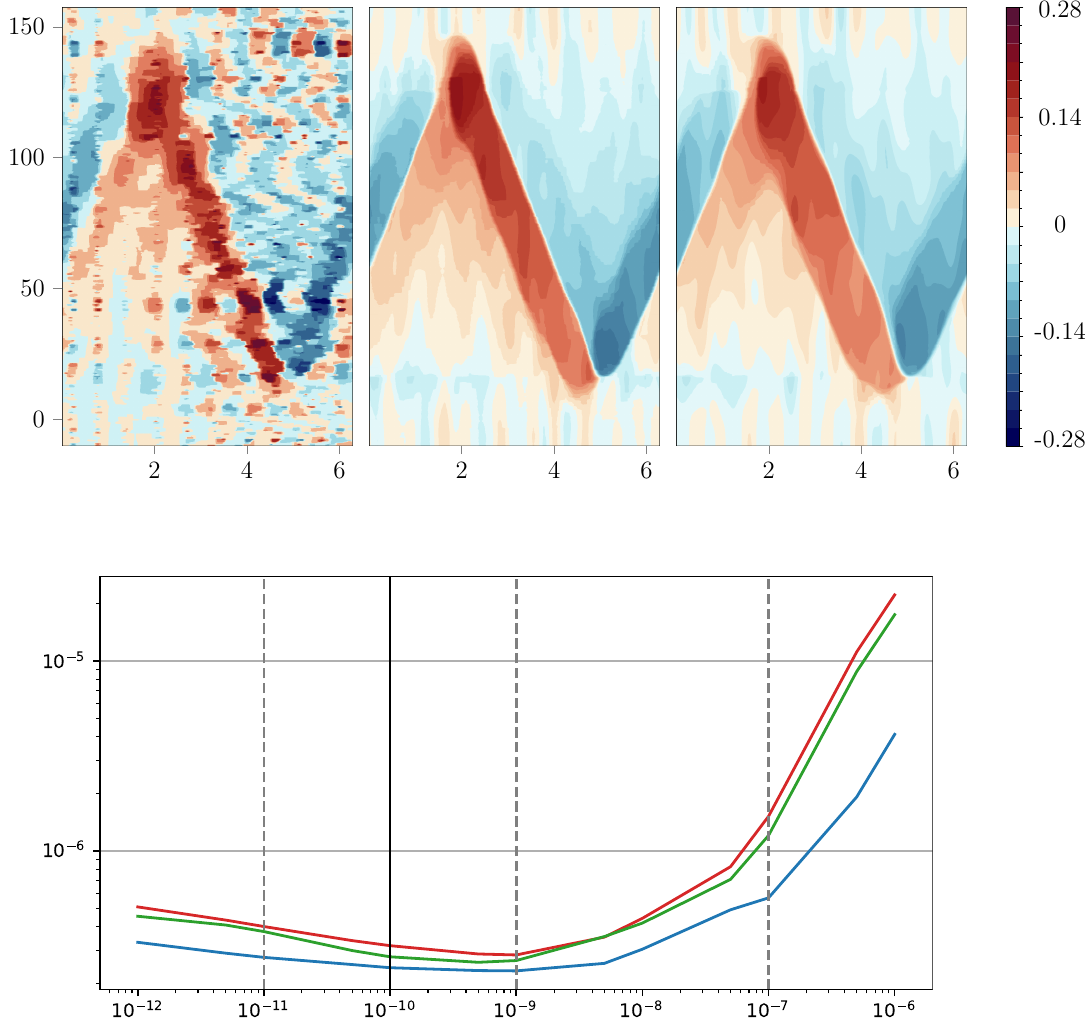}};
    \node[inner sep=0pt, rotate=90] at (-6.2,3.2) 
    {{\tiny Time $t$ [ms]}};
    
    \foreach \x in {0,...,2}
    {\node[inner sep=0pt] at (-3.8+\x*3.4,0.22){{\tiny Angle [rad]}};}

    \foreach \x/\y in {0/$\lambda_\gamma=\lambda_t=$1e-11,1/$\lambda_\gamma=\lambda_t=$1e-9,2/$\lambda_\gamma=\lambda_t=$1e-7}
    {\node[inner sep=0pt] at (-3.8+\x*3.4,5.8){{\small \y}};}

    \node[inner sep=0pt, rotate=90] at (-6.1,-3) 
    {{\small Validation Error}};
    \node[inner sep=0pt] at (-0.5,-6) 
    {{\small $\lambda_\gamma,\lambda_t$}};
    \node[inner sep=0pt] at (-0.5,-0.3) 
    {{Generalized Cross Validation}};

    \draw  (4.4,-1.2) rectangle (6.7,-4.95);
    \node[](S) at (4.5,-1.5){};
    \node[] (A) [right=0.4cm of S]{$6.7\%$};
    \node[] (B) [below right=0.4cm and 0.4cm of S]{$20\%$};
    \node[] (C)  [below right=0.4cm*3 and 0.4cm of S]{$33.3\%$};
    \node[] (D) [below right=0.4cm*5 and 0.4cm of S]{Optimal};
    \node[] (E) [below right=0.4cm*7 and 0.4cm of S]{Examples};
    \node[](S1) [left=0.4cm of B]{};
    \node[](S2) [left=0.4cm of C]{};
    \node[](S3) [left=0.4cm of D]{};
    \node[](S4) [left=0.4cm of E]{};

    \draw [draw=green] (S)[right] -- (A)[left];
    \draw [draw=red] (S1)[right] -- (B)[left];
    \draw [draw=blue] (S2)[right] -- (C)[left];
    \draw [draw=black] (S3)[right] -- (D)[left];
    \draw [draw=gray,dashed] (S4)[right] -- (E)[left];
    
    %\draw[line width=0.5mm] (-6.2,-0.1) -- (5,-0.1);
\end{tikzpicture}}
\caption{Generalized cross-validation with $50$dB noise for different percentages of observations $\{6.7\%, 20\%, 33.3\%\}$ and varying regularization parameters (uniform in space and time). The best reconstruction according to GCV, as well as an under- and an over-regularized reconstruction, are visualized.}
\label{fig:gcv}
\end{figure}
The results of the generalized cross-validation reveal that the same regularization parameters are selected across all validation data percentages, although they are slightly larger than the optimal values chosen in this study. Identifying the best parameter choices without ground truth for comparison is a non-trivial task and one that is still subject to various studies. Nonetheless, the close alignment of the GCV-selected regularization parameter with the optimal one serves as proof of the effectiveness of this approach for parameter selection.

\subsubsection{Three-Dimensional Rabbit Model}
In this part, we discuss the results for reconstructing the electrical activation on the rabbit dataset for $70$, $50$, and $20\db$ SNR levels.
All levels are evaluated for different errors in \Cref{tab:3D_results}.
Reconstructions for $50\db$ SNR are visualized in \Cref{fig:3D_noise_50} with extracellular potentials over time at a random epicardial node shown in \Cref{fig:3D_noise_ecg}.

\begin{figure}[htp] \centering{
\begin{tikzpicture}
    \node[inner sep=0pt] (A) at (0,0){\includegraphics[clip, trim=0.5cm 0.2cm 0cm 0.2cm,width=0.98\textwidth]{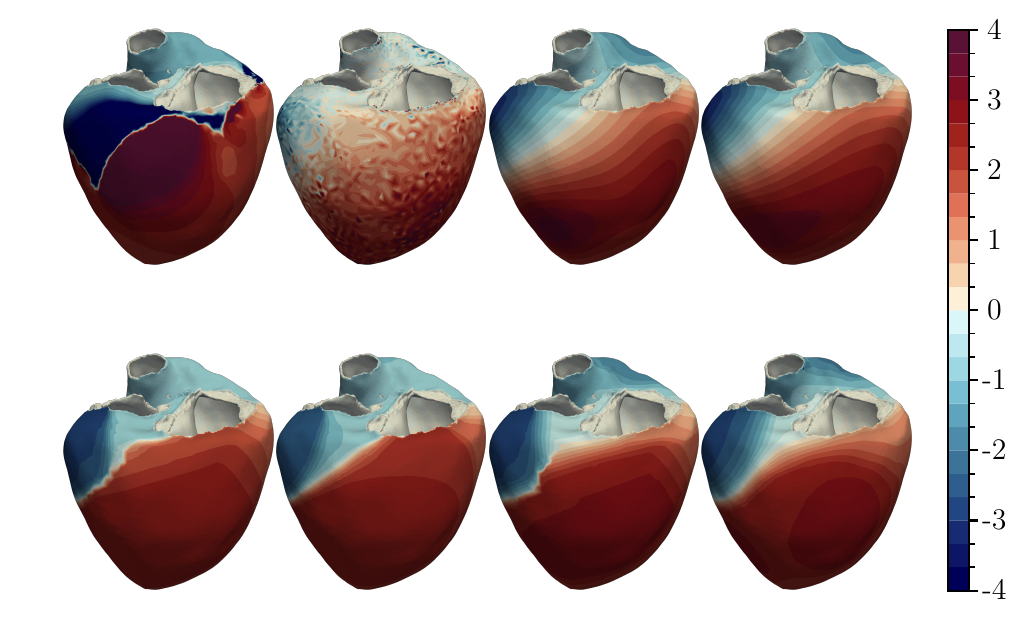}};
\foreach \x/\y in {0/$\GT$,1/$\Tz$,2/$\To$,3/$\TTo$}
    {\node[inner sep=0pt] at (-4.6+\x*2.75,3.9) 
    {\y};}
\foreach \x/\y in {0/$\TVSo$,1/$\TVSt$,2/$\TVSTo$,3/$\TVSTt$}
    {\node[inner sep=0pt] at (-4.6+\x*2.75,3.9-4.1) 
    {\y};}
\node[inner sep=0pt, rotate=90] at (6.4,0) 
    {extracellular potential (capped)};
\end{tikzpicture}}
\caption{Epicardial potential displayed at time $t=52.5ms$ for the reference function ($\GT$) and different reconstruction methods for the front of the rabbit heart with $50\db$ Gaussian noise on the BSPM.}
\label{fig:3D_noise_50}
\end{figure}

\begin{figure}[htp] \centering{
\begin{tikzpicture}
    \node[inner sep=0pt] (A) at (0,0){\includegraphics[clip, trim=0.2cm 0.2cm 0.2cm 0.3cm,width=0.98\textwidth]{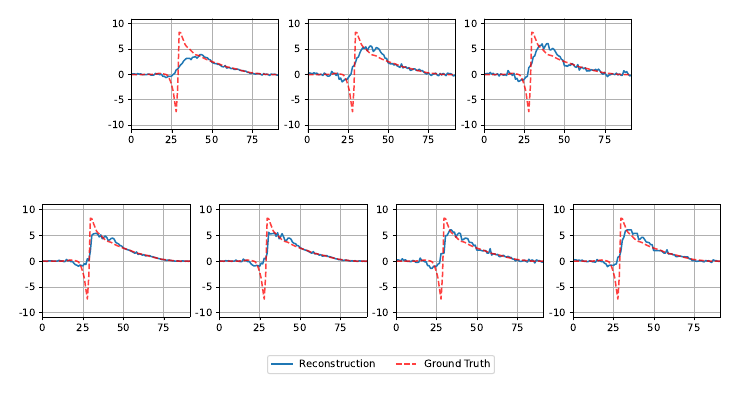}};
\node[inner sep=0pt, rotate=90] at (-4.8,2.3) 
    {{\tiny ECG [mV]}};
    \node[inner sep=0pt, rotate=90] at (-6.4,-1.05) 
    {{\tiny ECG [mV]}};

    \foreach \x/\y in {0/$\Tz$,1/$\To$,2/$\TTo$}
    {\node[inner sep=0pt] at (-2.9+\x*3.1,3.5) 
    {\y};}

    \foreach \x/\y in {0/$\TVSo$,1/$\TVSt$,2/$\TVSTo$,3/$\TVSTt$}
    {\node[inner sep=0pt] at (-4.5+\x*3.15,3.5-3.2) 
    {\y};}

    \foreach \x in {0,...,2}
    {\node[inner sep=0pt] at (-2.9+\x*3.15,0.8){{\tiny Time $t$ [ms]}};}
    
    \foreach \x in {0,...,3}
    {\node[inner sep=0pt] at (-4.5+\x*3.15,0.8-3.3){{\tiny Time $t$ [ms]}};}
\end{tikzpicture}}
\caption{Three-dimensional heart potential (ECG [mV]) with $50\db$ Gaussian noise added on the BSPM reconstructed by different methods over time (Time t [ms]) evaluated on an epicardial node of the corresponding finite element discretization.}
\label{fig:3D_noise_ecg}
\end{figure}

\begin{table}[htbp]
    \centering
    \label{tab:3D_results}
    \begin{scriptsize}
    \begin{tabular}{c|c|c|c|c|c|c|c|c}
     $\db$ & \diagbox[width=5.6em]{Error}{Reg.} & $\Tz$ & $\To$ & $\TTo$ & $\TVSo$& $\TVSt$& $\TVSTo$& $\TVSTt$\\
        \hline
          \multirow{4}*{70}& $\Vh \; \downarrow$ & $932.85$ & $813.21$ & $808.83$ & $795.78$ & $\mathit{791.19}$ & $792.03$ & $\mathbf{786.08}$\\
          & RE $\downarrow$ & $0.913$ & $0.77$ & $0.765$ & $0.754$ & $\mathit{0.75}$ & $0.752$ & $\mathbf{0.749}$\\
          & CC $\uparrow$ & $0.446$ & $0.638$ & $0.644$ & $0.656$ & $\mathit{0.66}2$ & $\mathit{0.66}$ & $\mathbf{0.664}$\\
         \cline{2-9}
         & $\lambda_\gamma,\lambda_t$ & 1e-10 & 1e-10 & 1e-10,1e-11 & 1e-10& 1e-10 & 1e-11,1e-11 & 1e-11,1e-11\\
        \hline 
        \multirow{4}*{50} & $\Vh \; \downarrow$ & $937.62$ & $827.98$ & $823.38$ & $819.54$ & $818.64$ & $\mathit{808.21}$ & $\mathbf{802}$\\
         & RE $\downarrow$ & $0.91$ & $0.783$ & $0.778$ & $0.773$ & $0.772$ & $\mathit{0.765}$ & $\mathbf{0.762}$\\
         & CC $\uparrow$ & $0.441$ & $0.622$ & $0.629$ & $0.634$ & $0.636$ & $\mathit{0.644}$ & $\mathbf{0.649}$\\
         \cline{2-9}
         & $\lambda_\gamma,\lambda_t$ & 1e-8 & 1e-8 & 1e-8,1e-9 & 1e-8& 1e-8 & 1e-9,1e-9 & 1e-9,1e-9\\
        \hline 
        \multirow{4}*{20} & $\Vh \; \downarrow$ & $961.43$ & $869.66$ & $857.83$ & $875.42$ & $869.83$ & $\mathit{854.46}$ & $\mathbf{849.65}$\\
         & RE $\downarrow$ & $0.938$ & $0.823$ & $\mathit{0.808}$ & $0.827$ & $0.821$ & $\mathbf{0.807}$ & $0.82$\\
         & CC $\uparrow$ & $0.387$ & $0.568$ & $\mathit{0.589}$ & $0.563$ & $0.571$ & $\mathbf{0.59}$ & $0.573$\\
         \cline{2-9}
         & $\lambda_\gamma,\lambda_t$ & 1e-6 & 1e-5 & 1e-5,1e-6 & 1e-6& 1e-5 & 1e-6,1e-6 & 1e-6,1e-6\\
    \end{tabular}
        \end{scriptsize}
    \caption{$\Vh$ error, relative error (RE), correlation coefficient (CC), and regularization parameters $(\lambda_\gamma,\lambda_t)$ for three-dimensional heart potential reconstruction of different regularization methods with respect to the ground truth function.}
\end{table}
The visual comparison in \Cref{fig:3D_noise_50} displays a massive difference between $\Tz$ and the remaining methods shown in the error comparison with $\Tz$ being the worst performing method.
Due to the highly ill-posedness of the problem, none of the methods reconstructs the epicardial potential without obvious visual and numerical errors.
Nevertheless, the improvement gained by total variation and added time regularization can be observed.
$L^{2,1}$ norm regularization improves the results regarding the $\Vh$ error with respect to $L^1$ norm regularization because of smoother transitions.
For increasing noise applied to the measurements, we observe less spatial curvature in \Cref{fig:3D_noise_70_20} in comparison to the ground truth due to the larger regularization parameters smoothing the epicardial potential.
In the $50\db$ case, the best performing $\TVSTt$ method achieves a relative gain to $\TTo$ of $2.6\%$ in the $\Vh$ error, $2\%$ in RE, and $3.2\%$ in CC.
The computing time amounts to $0.103$min for $\Tz,$ $0.115$min for $\To,$ $0.196$min for $\TTo$, at most $132.17$min for $\TVSo$ and $\TVSTo$, and at most $211.05$min for $\TVSt$ and $\TVSTt$.
For the $\Vh$ error, $\TVSTt$ produces the best reconstructions across all noise levels evaluated in this study.
Nevertheless, the $L^1$-norm space-time total variation method can be superior for different errors since we optimize only regarding the $\Vh$ error, and the other metrics are discrete vectors of nodal values-based metrics in contrast to the finite element discretized $\Vh$ metric.

For both 2D and 3D applications, introducing only a small amount of thorax measurement noise increases the optimal regularization parameter, which is expected because a small perturbation in the data corresponds to a large perturbation in the reconstruction.
Joint spatiotemporal smoothing of the function results in improved results due to the temporal continuity of the epicardium potential, information that is not exploited in spatial regularization methods.
Comparing the two- and three-dimensional results displays huge differences in the performance of $\Tz$, which is caused by the complexity of the meshes and the corresponding mass matrices.
Since the two-dimensional synthetic model contains a circular epicardium divided uniformly into affine finite elements, the impact of the mass matrix for every node is equal.
The three-dimensional rabbit heart geometry is much more complex and admits uneven smoothing on each epicardium node, resulting in a noise-perturbed reconstruction.
We can also observe less benefit for the temporal regularization, comparing the two- and three-dimensional results due to the highly ill-posedness of the geometrically more complex rabbit model.
The small regularization parameters for small noise levels indicate the necessity of a better-suited regularizer.

\begin{figure}[htp] \centering{
\begin{tikzpicture}
    \node[inner sep=0pt] (A) at (0,0){\includegraphics[clip, trim=0.5cm 0.2cm 0cm 0.2cm,width=0.93\textwidth]{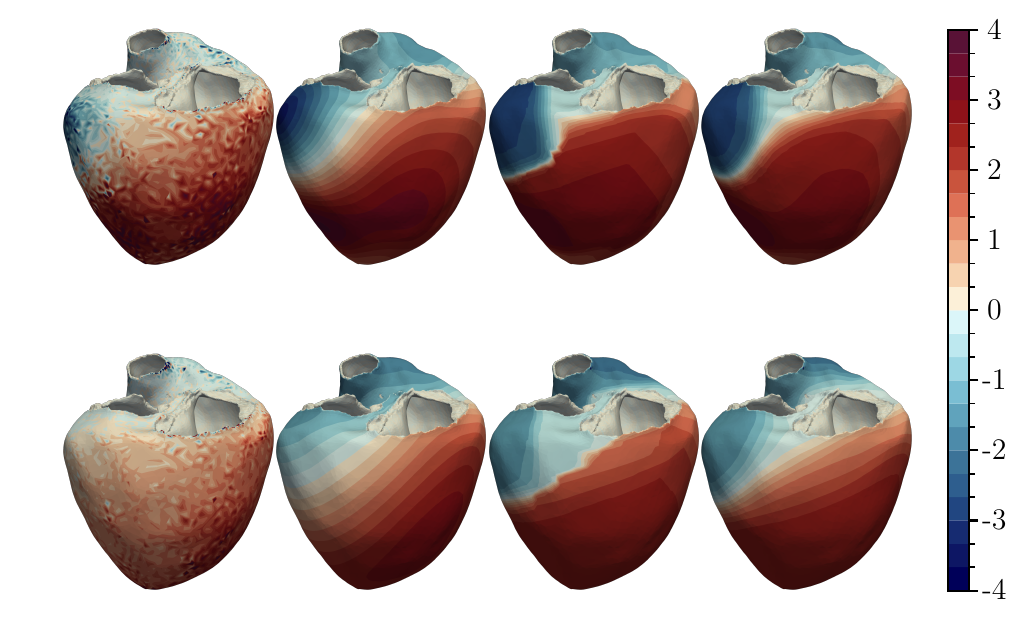}};
\foreach \x/\y in {0/$\Tz$,1/$\TTo$,2/$\TVSTo$,3/$\TVSTt$}
    {\node[inner sep=0pt] at (-4.45+\x*2.65,3.7) 
    {\y};}
\node[inner sep=0pt, rotate=90] at (6.1,0) 
    {extracellular potential (capped)};
\foreach \a/\b in {0/70,1/20}{
    \node[inner sep=0pt, rotate=90] at (-6.3,2.3-4.4*\a) 
        {$\b \db$};}
    \draw[line width=0.5mm] (-5.7,0.05) -- (5,0.05);
\end{tikzpicture}}
\caption{Three-dimensional heart potential reconstruction at time $t=52.5ms$ with ground truth function ($\GT$) for $70$ and $20\db$ white Gaussian noise.}
\label{fig:3D_noise_70_20}
\end{figure}

\subsection{Increasing Number of Electrodes}
Since the reconstruction of the epicardial potential will improve with more information by an increasing number of body surface measurements, we evaluate the proposed methods in this scenario to demonstrate the gain of the total variation-based regularization for a preferable body surface signal and epicardial node ratio.
The computations are performed with randomly chosen torso electrodes excluding the arms of the rabbit. 
The choice of regularization parameter is according to the optimal results with $32$ electrodes.
Even though the optimal hyperparameter values differ slightly for different numbers of electrodes, we chose the same parameter for all methods to guarantee comparability.
Because the optimal time regularization parameter for total variation decreases by consistent space regularization parameter, we allow for anisotropy.
\Cref{fig:3D_electrode_increase} displays the $\Vh$ error, RE, and CC for all reconstructions by an increasing number of electrodes.
A visual comparison with $500$ body surface measurements is displayed in  \Cref{fig:500_electrodes}.
\begin{figure}[htp] \centering{
\begin{tikzpicture}
    \node[inner sep=0pt] (A) at (0,0){\includegraphics[clip, trim=0cm 0.2cm 0cm 0cm,width=0.96\textwidth]{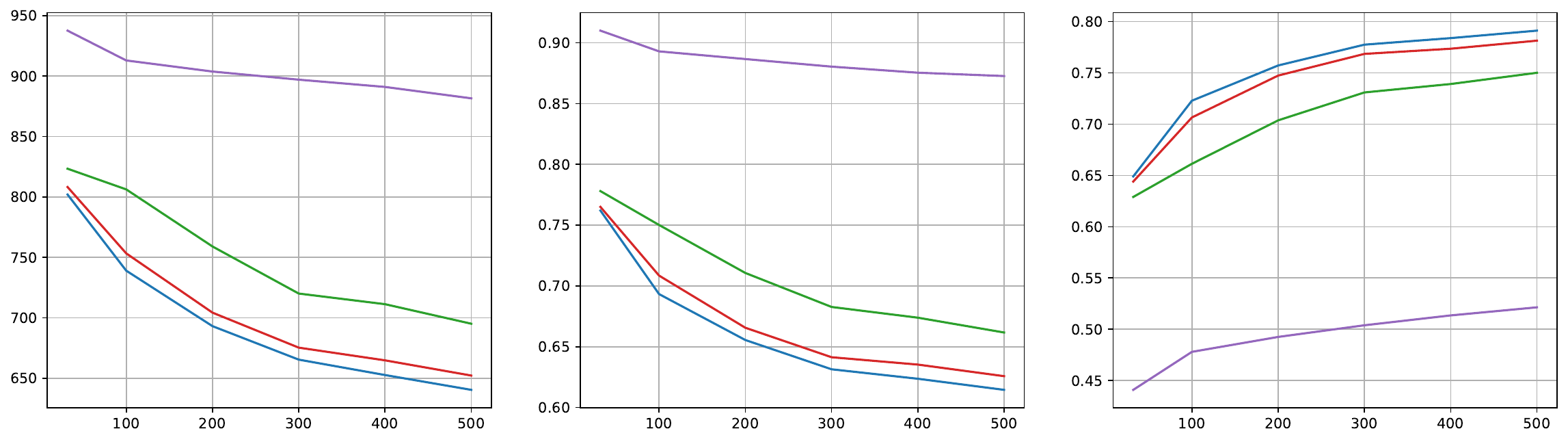}};
\node[inner sep=0pt, rotate=90] at (-6.3,0) {\small $\Vh$ Error};
\node[inner sep=0pt, rotate=90] at (-2.1,0) {\small RE};
\node[inner sep=0pt, rotate=90] at (2.125,0) {\small CC};

\node[inner sep=0pt] at (-4,-2) {\small Electrodes};
\node[inner sep=0pt] at (0.3,-2) {\small Electrodes};
\node[inner sep=0pt] at (4.5,-2) {\small Electrodes};

\draw  (-4.1,-2.35) rectangle (3.4,-2.85);
\node[](S) at (-4,-2.6){};
\node[] (A) [right=0.4cm of S]{$\Tz$};
\node[] (B) [right=0.4cm of A]{$\TTo$};
\node[] (C) at (0.7,-2.55) {$\TVSTo$};
\node[] (C2) [right=0.4cm of B]{};
\node[] (C3) [right=1.8cm of B]{};
\node[] (D) [right=0.4cm of C]{$\TVSTt$};
\node[] (D2) [right=0.4cm of C3]{};
\draw [draw=purple] (S)[right] -- (A)[left];
\draw [draw=green] (A)[right] -- (B)[left];
\draw [draw=red] (B)[right] -- (C2)[left];
\draw [draw=blue] (C3)[right] -- (D2)[left];
\end{tikzpicture}}
\caption{$\Vh$ error, RE, and CC for different rabbit heart reconstructions with $50\db$ Gaussian noise evaluated for increasing body surface measurements. The chosen regularization parameters are $\lambda_\gamma=1e^{-8}$ for $\Tz$, $\lambda_\gamma,\lambda_t=1e^{-9},(1e^{-9};1e^{-10})$ for $\TTo$, and $\lambda_\gamma,\lambda_t=1e^{-9},(1e^{-9};1e^{-10})$ for $\TVSTa$ with $\alpha\in\{1,2\}$.}
\label{fig:3D_electrode_increase}
\end{figure}

\begin{figure}[htp] \centering{
\begin{tikzpicture}
   \node[inner sep=0pt] (A) at (0,0){\includegraphics[clip, trim=0.5cm 0.2cm 0cm 0.2cm,width=0.973\textwidth]{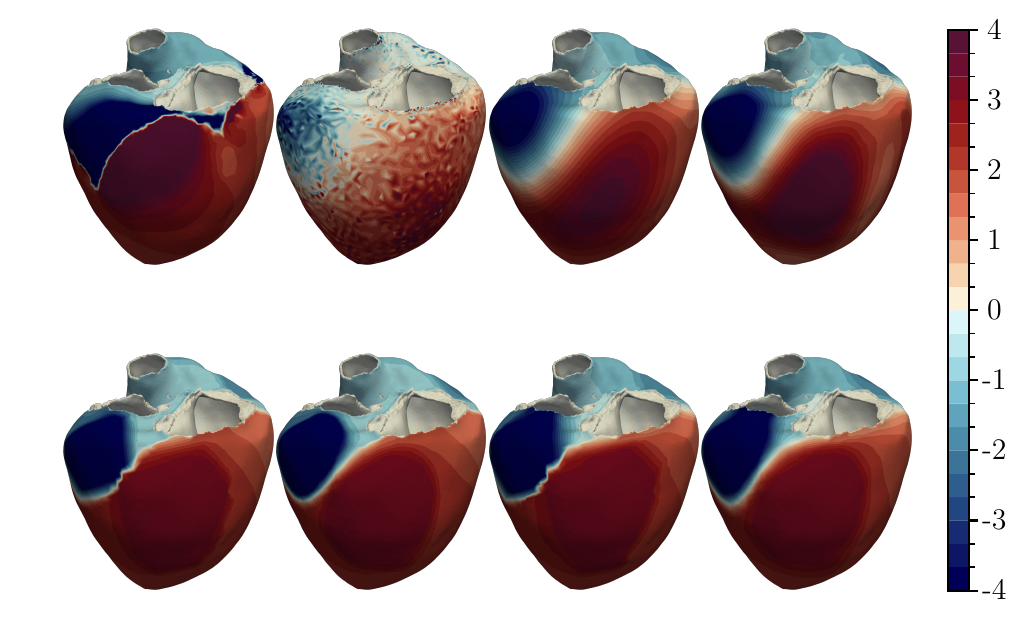}};
\foreach \x/\y in {0/$\GT$,1/$\Tz$,2/$\To$,3/$\TTo$}
    {\node[inner sep=0pt] at (-4.5+\x*2.7,3.9) 
    {\y};}
\foreach \x/\y in {0/$\TVSo$,1/$\TVSt$,2/$\TVSTo$,3/$\TVSTt$}
    {\node[inner sep=0pt] at (-4.5+\x*2.7,4-4.3) 
    {\y};}
\node[inner sep=0pt, rotate=90] at (6.4,0) 
    {extracellular potential (capped)};
\end{tikzpicture}}
\caption{Rabbit database epicardial potential at time $t=52.5ms$ with $500$ body surface measurements for the reference function ($\GT$) and different reconstruction methods with $50\db$ Gaussian noise.}
\label{fig:500_electrodes}
\end{figure}
The plot of decreasing $\Vh$ error and RE and increasing CC for all shown reconstruction methods confirms the assumption that the regularization proposed in this paper leads to improvements by an even larger margin regarding Tikhonov methods if more measurement information is obtained.
Furthermore, in the visual comparison of all algorithms, the sharp edges enforced by total-variation regularization are visible for all $\TVSa$ and $\TVSTa$ methods.
$\TVSTt$ reconstructs the epicardial potential best for all body surface measurement numbers.
The relative gain with $500$ body surface measurements of the best performing $\TVSTt$ method in comparison to $\TTo$ is $9.5\%$ for the $\Vh$ error, $6.4\%$ for the RE, and $8.5\%$ for CC.

\section{Conclusions}
This work introduced novel total variation-based regularization methods in a finite element setting to reconstruct the epicardial potential from body surface measurements.
After stating the mathematical model of the spatiotemporal inverse problem in ECGI, we introduced different discretizations to improve total variation-based regularization.
The resulting methods enhance reconstructions compared to state-of-the-art Tikhonov regularization since the regularizer captures the characteristics of the heart potential function more effectively.
We tested our approaches in a two- and three-dimensional database with simulated epicardial potential and showed their benefits of joint spatiotemporal regularization.

The non-linear optimization problem for total variation-based regularization is computationally more demanding due to the non-differentiability of the regularizer near zero and the high-dimensional dual variable in the iterative minimization algorithm proposed in this paper.
However, the improvements in accuracy achieved by total variation-based reconstruction and dependence on the resolution of the problem justify the computational complexity.

Still, the clinical application demands much more improvement regarding reconstructing the inverse problem in electrocardiographic imaging.
In conclusion, the spatiotemporal total variation-based approach is limited by the computational complexity of high-dimensional datasets.
Future work will further improve the regularization of the inverse problem and focus on real human heart geometry and non-synthetic models.

\section*{Acknowledgements}
Manuel Haas and Thomas Grandits share the first authorship.

%\section*{Conflict of interest}
%The authors declare that they have no conflict of interest.

\bibliographystyle{siamplain}
\bibliography{references.bib}
\end{document}